\documentclass[10pt]{amsart}
\usepackage{amssymb,amsmath,amsfonts,amsthm,enumerate}
\usepackage[normalem]{ulem}
\usepackage{marginnote}

\DeclareFontFamily{OML}{rsfs}{\skewchar\font'177}
\DeclareFontShape{OML}{rsfs}{m}{n}{ <5> <6> rsfs5 <7> <8> <9> rsfs7
  <10> <10.95> <12> <14.4> <17.28> <20.74> <24.88> rsfs10 }{}
\DeclareMathAlphabet{\mathfs}{OML}{rsfs}{m}{n}

\usepackage{hyperref}
\RequirePackage[dvipsnames]{xcolor} 
\definecolor{halfgray}{gray}{0.55} 
\definecolor{webgreen}{rgb}{0,0.5,0}
\definecolor{webbrown}{rgb}{.6,0,0} \hypersetup{%
  colorlinks=true, linktocpage=true, pdfstartpage=3,
  pdfstartview=FitV,%
  breaklinks=true, pdfpagemode=UseNone, pageanchor=true,
  pdfpagemode=UseOutlines,%
  plainpages=false, bookmarksnumbered, bookmarksopen=true,
  bookmarksopenlevel=1,%
  hypertexnames=true,
  pdfhighlight=/O,
  urlcolor=webbrown, linkcolor=RoyalBlue,
  citecolor=webgreen, 
  pdftitle={},%
  pdfauthor={},%
  pdfsubject={2000 MAthematical Subject Classification: Primary:},%
  pdfkeywords={},%
  pdfcreator={pdfLaTeX},%
  pdfproducer={LaTeX with hyperref}%
}

\newtheorem{theorem}{Theorem}[section]
\newtheorem{lemma}[theorem]{Lemma}
\newtheorem{proposition}[theorem]{Proposition}
\newtheorem{corollary}[theorem]{Corollary}
\theoremstyle{definition}
\newtheorem{definition}[theorem]{Definition}

\newtheorem{remark}[theorem]{Remark}

\newtheorem{ltheorem}{Theorem} 

\def\real{\mathbb{R}}
\def\complex{\mathbb{C}}

\def\natural{\mathbb{N}}
\def\field{\mathbb{K}}

\def\d{\operatorname{dist}}
\def\dim{\operatorname{dim}}

\def\car{\operatorname{1}}
\def\per{\operatorname{per}}

\def\cV{\mathfs R}

\def\cN{\mathcal{N}}

\def\hA{\widehat{A}}
\def\hrho{\widehat{\rho}}

\def\GL{GL(d,\field)}
\def\GLC{GL(d,\complex)}

\def\hA{\widehat{A}}
\def\hB{\widehat{B}}
\def\hF{\widehat{F}}
\def\hsigma{\widehat{\sigma}}
\def\hSigma{\widehat{\Sigma}}
\def\hDelta{\widehat{\Delta}}

\def\hpi{\widehat{\pi}}
\def\hx{\underline R}
\def\hz{\underline U}

\def\hp{\underline P}
\def\hy{\underline S}

\def\hm{\widehat{m}}
\def\hmu{\widehat{\mu}}
\def\hnu{\widehat{\nu}}

\def\hpsi{\widehat{\psi}}
\def\hphi{\widehat{\phi}}

\def\tR{\widetilde{R}}

\newcommand{\norm}[1]{{\left\lVert  #1  \right\rVert}}
\newcommand{\abs}[1]{{\left\lvert  #1  \right\rvert}}

\title[Simplicity of Lyapunov spectrum]
{Simplicity of Lyapunov spectrum for linear cocycles  over non-uniformly hyperbolic systems}

\author[Lucas Backes]{Lucas Backes}
\address{Lucas Backes, Departamento de Matem\'atica, Universidade Federal do Rio Grande do Sul, Av. Bento Gon\c{c}alves 9500, CEP 91509-900, Porto Alegre, RS, Brazil.}
\email{lhbackes@impa.br }

\author[Mauricio Poletti]{Mauricio Poletti}

\address{Mauricio Poletti, LAGA -- Universit\'e Paris 13, 99 Av. Jean-Baptiste Cl\'ement, 93430 Villetaneus, France.}
\email{mpoletti@impa.br}
\author[Paulo Varandas]{Paulo Varandas \newline With an appendix by Yuri Lima}
\address{Paulo Varandas, Departamento de Matem\'atica, Universidade Federal da Bahia\\
Av. Ademar de Barros s/n, 40170-110 Salvador, Brazil}
\email{paulo.varandas@ufba.br}

\address{Yuri Lima, Departamento de Matem\'atica, Universidade Federal do Cear\'a (UFC), Campus do Pici,
Bloco 914, CEP 60455-760. Fortaleza -- CE, Brasil}
\email{yurilima@gmail.com}

\subjclass[2010]{37H15,37D30,37D25}
\keywords{Lyapunov exponents, non-uniform hyperbolicity, linear cocyles}

\begin{document}

\begin{abstract}
We prove that generic fiber-bunched and H\"older continuous \mbox{linear} cocycles over a non-uniformly hyperbolic system endowed with a $u$-Gibbs \mbox{measure} have simple Lyapunov spectrum. This gives an affirmative answer to a conjecture proposed by Viana in the context of fiber-bunched cocycles. 
\end{abstract}

\maketitle


\section{Introduction}

The notion of (uniform) hyperbolicity was introduced in the context of dynamical systems by Smale \cite{Sm67}, and since then has played a major rule in this area of research. This notion is expressed in terms of (uniform) rates of contraction and expansion by the dynamics along complementary directions. At first, it was conjectured that 
uniform hyperbolicity is quite frequent among all dynamical systems. Newhouse then proved that this was not the case: he exhibited a $C^2$-open set of diffeomorphisms on the $2$-sphere where none of its elements are
hyperbolic \cite{New70}. 

In order to describe the majority of dynamical systems, weaker notions of hyperbolicity were introduced and have been intensively studied. These include partially hyperbolic and non-uniformly hyperbolic dynamical systems. The notion of non-uniform hyperbolicity is defined in terms of \textit{Lyapunov exponents}: a diffeomorphism
is non-uniformly hyperbolic if it has no zero Lyapunov exponents. Lyapunov exponent measure the \textit{asymptotic} rates of contraction and expansion along directions and are one of the most fundamental notions in dynamical systems. They have received a great deal of attention in the last decades, and they are the focus of the present work. Among many interesting questions that can be posed about them, we cite the following:
\begin{enumerate}[$\circ$]
\item What are the regularity properties of Lyapunov exponents?
\item How frequently a system has \emph{at least one} non-zero Lyapunov exponent?
\item How frequently a system has Lyapunov exponents  \emph{all different from zero}?
\item How frequently a system has Lyapunov exponents \textit{all different}?
\end{enumerate}
The context of \emph{linear cocycles} has provided a fruitful playground for addressing these questions,
since it allows to detach the underlying dynamics from an action, induced by it, on a vector space.
In this context the abundance of non-uniform hyperbolicity, foreseen since the pioneering works of Furstenberg \cite{Fur63}, Guivarc'h and Raugi \cite{GR86}, Golâdsheid and Margulis \cite{GM89}
on random i.i.d. products of matrices, was later extended by Bonatti and Viana \cite{BoV04}, Viana~\cite{Almost},
Avila and Viana \cite{AvV1} to include a much broader class of (H\"older continuous) cocycles over hyperbolic maps. More recently, Matheus, M\"oller and Yoccozz \cite{MMY15} considered the Kontsevich-Zorich cocycle, Poletti and Viana \cite{PoV16} established a criterion for simplicity of the Lyapunov spectrum for cocycles over partially hyperbolic maps, and Bessa et al ~\cite{BBCMVX} considered the top Lyapunov exponent for linear cocycles on semisimple Lie groups over non-uniformly hyperbolic diffeomorphisms.
One should mention that the coincidence of all Lyapunov exponents (in opposition to the abundance of non-uniform hyperbolicity) occurs for generic continuous cocycles over ergodic automorphisms,
as proved by Bochi \cite{Boc02}. In this note, we are interested in the last question:
how frequently a system has Lyapunov exponents \textit{all different}? When this happens, we say
that the Lyapunov spectrum is \textit{simple}. Our main result is the following.

\begin{theorem}
Generic fiber-bunched linear cocycles over a non-uniformly hyperbolic system have simple Lyapunov spectrum.
\end{theorem} 

Since we need some preliminary definitions, the precise statement of our result is at the end of Section~\ref{definitions} after some preliminary definitions. It provides an affirmative answer to a
conjecture of Viana in \cite{Almost}, p. 648, in the context of fiber-bunched cocycles:
the set of linear cocycles whose Lyapunov exponents are all different contains an open and dense set
of fiber-bunched cocycles. We remark that other similar questions have also attracted the attention of 
the community and have already been answered in some specific contexts, see
\cite{BBB, BP, Be2,BeVar, BGV03, BoV04, DK_b,  Po16,  LLE} and references therein.
We also remark that for the \textit{dynamical cocycle}, corresponding to the cocycle
$Df$ over the diffeomorphism $f$, fewer results are known.

\section{Definitions and Statements} \label{definitions}

In this section we introduce some preliminary notions and provide the precise statement of our main result.
Let $M$ be a closed smooth manifold, $f:M\to M$ a $C^{1+\beta}$ diffeomorphism and $\mu$ an ergodic $f$-invariant measure. 

\subsection{Linear cocycles and Lyapunov exponents} Given an integer $d\ge 1$ and $\mathbb K= \mathbb R$
or $\mathbb C$, the \emph{linear cocycle generated by a matrix-valued map $A:M\rightarrow \GL$ over $f$} is the (invertible) map $F_A:M\times \field ^d\rightarrow M\times \field ^d$ defined by
\begin{equation}
\nonumber F_A\left(x,v\right)=\left(f(x),A(x)v\right).
\end{equation}
Its iterates are ${F}^n_A\left(x,v\right)=\left(f^n(x),A^n(x)v\right)$, where
\begin{equation*}\label{def:cocycles}
A^n(x)=
\left\{
	\begin{array}{ll}
		A(f^{n-1}(x))\cdots A(f(x))A(x)  & \mbox{if } n>0 \\
		{\rm Id} & \mbox{if } n=0 \\
		A(f^{n}(x))^{-1}\cdots A(f^{-1}(x))^{-1}& \mbox{if } n<0 .\\
	\end{array}
\right.
\end{equation*}
Sometimes we denote this cocycle by $(f,A)$ or simply by $A$, when there is no risk of ambiguity. 
The \emph{projectivized cocycle} $f_A : M\times \mathbb{P}^{d-1}(\field) \rightarrow M\times \mathbb{P}^{d-1}(\field)$ is
$$
f_A\left(x,v\right)=\left(f(x),\frac{A(x)v}{\|A(x)v\|}\right).
$$
A natural example of linear cocycle is given by the \textit{derivative cocycle}:
the cocycle generated by $A(x)=Df(x)$ over $f$.

When $\log\norm{A}$ and $\log\norm{ A^{-1}}$ are both integrable, 
a famous theorem of Oseledets~\cite{Ose68} guarantees the existence of a full $\mu$-measure set $\mathcal{R} (\mu)\subset M$, whose points are called \emph{$\mu$-regular}, such that for every $x\in \mathcal{R} (\mu)$ there exist real numbers $\lambda_1 \left(A, x\right)>\cdots >\lambda_k\left(A, x\right)$ and a direct sum decomposition $\field ^d=E^{1,A}_{x}\oplus \cdots \oplus E^{k,A}_{x}$ such that
\begin{displaymath}
A(x)E^{i,A}_{x}=E^{i,A}_{f(x)}  \textrm{  and  }
\lambda _i(A, x) =\lim _{n\rightarrow \infty} \dfrac{1}{n}\log \| A^n(x)v\| 
\end{displaymath}
for every non-zero $v\in E^{i,A}_{x}$ and $1\leq i \leq k$. Moreover, since $\mu$ is ergodic, the \textit{Lyapunov exponents} $\lambda _i(A, x)$ are constant on a full $\mu$-measure subset of $M$ (and thus we denote it just by $\lambda_i(A,\mu)$) as well as the dimensions of the \textit{Oseledets subspaces} $E^{i,A}_{x}$.
The dimension of $E^{i,A}_{x}$ is called the \emph{multiplicity} of $\lambda_i(A,\mu)$. We say that the cocycle $(f,A)$ has \textit{simple Lyapunov spectrum} with respect to the measure $\mu$ if every Lyapunov exponent has multiplicity one. This means that $A$ has $d$ distinct Lyapunov exponents.

\subsection{Non-uniformly hyperbolic systems and Gibbs states}\label{sec: non unif hyp}

An $f$-invariant measure $\mu$ is said to be \textit{hyperbolic} if all Lyapunov exponents $\{\lambda _i(Df, \mu) \}_{i=1}^{k}$ are non-zero. When this happens, $(f,\mu)$ is called \textit{non-uniformly hyperbolic}. 
Given $\chi>0$, $\mu$ is called \textit{$\chi$-hyperbolic} if 
$0< \chi < \min \{ | \lambda _i(Df, \mu) | \colon 1\leq i\leq k\}$.
Non-uniform hyperbolicity implies the existence of a very rich geometric structure of the dynamics of $f$,
given by stable and unstable manifolds in the sense of Pesin (see \cite{BaP07}):
there exists a full $\mu$-measure set $H(\mu) \subset M$
so that through every point $x\in H(\mu)$ there exist $C^1$ embedded disks $W^s_{\text{\rm loc}}(x)$ and $W^u_{\text{\rm loc}}(x)$, called \textit{local stable and unstable sets} at $x$, such that
\begin{enumerate}[$\circ$]
\item $W^s_{\text{\rm loc}}(x)$ is tangent to $E^s_x$ and $W^u_{\text{\rm loc}}(x)$ is tangent to $E^{u}_x$;
\item given $0<\tau _x <\min _{1\leq i\leq k}| \lambda _i(Df,\mu)|$ there exists $C_x>0$ such that 
$$
\left\{
\begin{array}{rl}
d(f^n(y),f^n(z))\leq C_xe^{-\tau _x n}d(y,z) & ,\forall y,z\in W^s_{\text{\rm loc}}(x),\forall n\geq 0,\\
d(f^{-n}(y),f^{-n}(z))\leq C_xe^{-\tau _x n}d(y,z) & ,\forall y,z\in W^u_{\text{\rm loc}}(x),\forall n\geq 0;
\end{array}
\right.
$$
\item $f(W^s_{\text{\rm loc}}(x))\subset W^s_{\text{\rm loc}}(f(x))$ and $f(W^u_{\text{\rm loc}}(x)) \supset W^u_{\text{\rm loc}}(f(x))$;
\item $W^s(x)=\bigcup _{n=0}^{\infty} f^{-n}(W^s_{\text{\rm loc}}(f^n(x)))$ and $W^u(x)=\bigcup _{n=0}^{\infty} f^{n}(W^u_{\text{\rm loc}}(f^{-n}(x)))$.
\end{enumerate}
Moreover, $W^s_{\text{\rm loc}}(x)$ and $W^u_{\text{\rm loc}}(x)$ depend measurably on $x$, as $C^1$ embedded disks, as well as the constants $\tau _x$ and $C_x$. By Lusin's theorem, we may find compact \textit{hyperbolic blocks} $\mathcal{H}(C,\tau)$, whose measure can be made arbitrarily close to 1 by increasing $C$ and decreasing $\tau$, such that in $\mathcal{H}(C,\tau)$ the sets $W^s_{\text{\rm loc}}(x)$ and $W^u_{\text{\rm loc}}(x)$ vary continuously,
$\tau _x>\tau \text{ and } C_x<C$.
In particular, the sizes of $W^s_{\text{\rm loc}}(x)$ and $W^u_{\text{\rm loc}}(x)$ are uniformly bounded from zero on $\mathcal{H}(C,\tau)$, as well as the angles between these disks.
The drawback on this argumentation is that  $\mathcal{H}(C,\tau)$ is in general not $f$-invariant.


\begin{definition}
An $f$-invariant measure $\mu$ is called a \emph{$u$-Gibbs state} (respectively, an \emph{$s$-Gibbs state}) if its disintegrations along unstable (respectively, stable) manifolds are absolutely continuous with respect to the Lebesgue measure.
\end{definition}

In the literature, $u$-states are also called \emph{SRB measures}. The simplest example
is given by volume preserving diffeomorphisms: the Lebesgue measure is both an $s$-Gibbs and
a $u$-Gibbs state.
The class of $u$-Gibbs measures is physically relevant: 
physical measures\footnote{An ergodic invariant probability measure is a \textit{physical measure}
if its basin of attraction has positive Lebesgue measure.}
of uniformly hyperbolic and of many partially hyperbolic attractors are $u$-Gibbs states, see e.g. 
\cite{AL,Bo75, PeS83} and references therein.
By \cite{Le84a,LY,LS82}, an $f$-invariant measure $\mu$ is a $u$-Gibbs state iff it
satisfies the \textit{entropy formula}
$$
h_\mu(f)=\int \sum_{\lambda_i(x)>0} \lambda_i(x)d\mu.
$$ 
Analogously, $\mu$ is a $s$-Gibbs state iff
$h_\mu(f)=\int \sum_{\lambda_i(x)<0} -\lambda_i(x)d\mu.$

\subsection{Fiber-bunched cocycles}
Given $r\in \mathbb N$ and $\alpha \in [0,1]$, let $C^{r,\alpha}(M, GL(d,\mathbb K))$ denote the set of 
$C^{r}$ maps $A: M \to GL(d,\mathbb K)$ such that $D^r A$ is $\alpha$-H\"older continuous.
Endow $C^{r,\alpha}(M, GL(d,\mathbb K))$ with the norm
$$
\|A\|_{r,\alpha}
	:= \max_{0\le j\le r} \sup_{x\in M} \| D^j A (x)\| + \sup_{x,y\in M\atop{x\neq y}} \frac{\| D^rA(x) -D^rA(y) \|}{d(x,y)^\alpha}.
$$
For $\chi>0$, let $\mu$ be an ergodic $\chi$-hyperbolic measure. 
We say that the cocycle $A\in C^{r,\alpha}(M, GL(d,\mathbb K))$  over 
$(f,\mu)$ is \emph{$\frac{\chi}{2}$-fiber-bunched} if there are constants $C>0$ and $\theta \in (0,1)$ such that
\begin{equation}\label{def:fiberbunching}
\| A^n(x)\| \| A^{n}(x)^{-1}\| (e^{-\frac{\chi}{2}})^{| n|  \alpha}\leq C \theta ^{| n|}
\end{equation}
for every $x\in M$ and $n\in \mathbb{Z}$. 
A simple remark is that the previous notion does not depend on the invariant measure $\mu$ but on the hyperbolicity constant $\chi>0$.  
Let $\mathcal B_\chi^{r,\alpha}(M)$ denote the set of $C^{r,\alpha}$ cocycles that are
$\frac\chi{2}$-fiber-bunched over $(f,\mu)$. It is a $C^0$-open subset of $C^{r,\alpha}(M, GL(d,\mathbb K))$.

\subsection{Main theorem} The main result of this work is that, for $u$-Gibbs measures,
a typical fiber-bunched cocycle has simple Lyapunov spectrum. More precisely, we have the following theorem.

\begin{ltheorem} \label{theo: main}
Let $f:M\to M$ be a $C^{1+\beta}$ diffeomorphism of a closed smooth Riemannian manifold, and let
$\mu$ be an ergodic $\chi$-hyperbolic $u$-Gibbs measure. For each $r\geq 0$ and $\alpha \geq 0$ with
$r+\alpha>0$, there exists an open and dense subset $\mathcal O \subset \mathcal B_\chi^{r,\alpha}(M)$ 
such that for any $A\in \mathcal O$, the cocycle $(f,A)$ has simple Lyapunov spectrum. 
\end{ltheorem}

It is clear that Theorem \ref{theo: main} above also applies for $s$-Gibbs measures. Indeed,
while on the one hand $\mu$ is $s$-Gibbs for $f$ iff it is $u$-Gibbs for $f^{-1}$, on the other hand a cocycle is fiber-bunched for $(f,\mu)$ iff it is fiber-bunched for $(f^{-1},\mu)$. Moreover, similar calculations to those of \cite{BoV04} imply that the complement $\mathcal B_\chi^{r,\alpha}(M) \setminus \mathcal O$ has infinite codimension, that is, it is locally contained in finite unions of closed submanifolds with arbitrarily large codimension. Theorem \ref{theo: main} should be compared to \cite{Almost}, where the author proved that an open and dense set of $C^{r,\alpha}$ cocycles over an ergodic $\chi$-hyperbolic measure with local product structure such that
each of the cocycles has at least one positive Lyapunov exponent.

At this point it would be interesting to exhibit non-trivial examples of open sets of cocycles which are not fiber-bunched and yet have simple Lyapunov spectrum. By non-trivial examples we mean, for instance, not having dominated decomposition.

The proof of Theorem \ref{theo: main} has two main ingredients: the first is the simplicity criterion
established by Avila and Viana \cite{AvV1}, and the second is the symbolic description of non-uniformly
hyperbolic diffeomorphisms given by Sarig \cite{Sa13} and Ben Ovadia \cite{Ova16}. 

We would also like to mention that our result can be extended to the case when the base dynamics is the class of billiards considered by Lima and Matheus in \cite{LM16}, since their symbolic description 
is the same as the ones given in \cite{Sa13,Ova16}.

\section{Preliminaries}

This section recalls some notions and results that will be used in the proof of Theorem \ref{theo: main}.

\subsection{Topological Markov shifts}\label{sec: Markov partitions}
The notation in the section is the same as in the Appendix.
Let $\mathfs G$ be a directed graph with a countable vertex set $\cV$, such that every vertex has at least one
ingoing and one outgoing edge. The \textit{topological Markov shift} associated to $\mathfs{G}$ is the 
pair $(\hSigma,\hsigma)$ where
$$\hSigma=\hSigma(\mathfs{G}):=\{\hx=\{R_i\}_{i\in\mathbb{Z}} \in{\cV}^\mathbb{Z}:
R_i \rightarrow R_{i+1},\forall i\in\mathbb{Z}\}$$ 
and $\hsigma:\hSigma\rightarrow\hSigma$ is the left-shift,
$\hsigma[\{R_i\}_{i\in\mathbb{Z}}]=\{R_{i+1}\}_{i\in\mathbb{Z}}$. Define a metric on $\Sigma$ by 
\begin{equation}\label{eq: adapted metric}
d[\{R_i\}_{i\in\mathbb{Z}},\{S_i\}_{i\in\mathbb{Z}}]:=\exp\left[-\tfrac{\chi}{2}\min\{| n| :\; R_n\neq S_n\}\right].
\end{equation}
With this metric, $\hSigma$ is a complete separable metric space and 
$\hsigma$ is a hyperbolic homeomorphism (with hyperbolicity constant $e^{\frac{\chi}{2}}$).
Furthermore, $\hSigma$ is compact iff $\mathfs{G}$ is finite, and it is
\emph{locally compact} iff every vertex of $\mathfs{G}$ has finite ingoing and outgoing degrees.
We define
$$\hSigma^\#:=\left\{\{R_i\}_{i\in\mathbb{Z}}\in\hSigma:\exists \; R,S\in\cV,\exists \; n_k, m_k\uparrow\infty\text{ s.t. } R_{n_k}=R\text{ and }R_{-m_k}=S\right\}.$$
This set contains all periodic point of $\hsigma$. Also, by the Poincar\'e recurrence theorem,
every $\hsigma$-invariant probability measure is supported on $\hSigma^\#$. 

The next result plays an important role in our proof. It was first established by Sarig for surface diffeomorphisms 
\cite{Sa13} and later generalized to any dimension by Ben Ovadia \cite{Ova16}.

\begin{theorem}[\cite{Sa13,Ova16}]\label{theo: Markov partitions} 
Let $f:M\to M$ be a $C^{1+\beta}$ diffeomorphism. For each $\chi>0$, there exists a locally compact topological Markov shift $(\hSigma,\hsigma)$ and a H\"older continuous map $\hpi:\hSigma\rightarrow M$ such that: \begin{enumerate}[{\rm (1)}]
    \item $\hpi\circ\hsigma=f\circ\hpi$.
    \item $\hpi[\hSigma^\#]$ has full measure for every $\chi$-hyperbolic measure.
    \item Every $x\in \hpi_\chi[\hSigma^\#]$ has finitely many pre-images in $\hSigma^\#$. More specifically:
    there is $\varphi:\mathfs R\to\mathbb N$ such that if $x=\hpi(\hx)$ with $R_n=R$ for infinitely many
    $n>0$ and $R_n=S$ for infinitely many $n<0$ then $\#\{\hy\in \hSigma^\#:\hpi(\hy)=x\}\leq \varphi(R)\varphi(S)$.
    \item For every $\chi$-hyperbolic measure $\mu$, there exists a $\hsigma$-invariant measure $\hmu$
    such that $\hpi_*\hmu=\mu$.
    \end{enumerate}
\end{theorem}

In Appendix \ref{appendix} we present a brief sketch of the proof of Theorem \ref{theo: Markov partitions}.
The proof not only provides the existence $(\hSigma,\hsigma)$ and $\hpi$, but it also implies
some additional properties of these objects, such as Lemma \ref{l.product} below. That is why
we decided, to maintain consistence and simplify the understanding of the reader,
to use the notations of \cite{Sa13,LS16}.

\begin{remark} \label{rem: pi is lipschitz}
It follows from \cite[Theorem 1.3.21]{Ova16} that the H\"older exponent of 
$\hpi$ is $ \frac{\chi}{2}$ with respect to the canonical metric
$d'(\hx,\hy)=\exp\left[-\min\lbrace \abs{n}:R_n\neq S_n\rbrace \right]$
(cf. proof of \cite[Proposition 1.3.20]{Ova16}). Therefore, $\hpi$ is Lipschitz with respect to the
metric defined by \eqref{eq: adapted metric}.
\end{remark}

\subsection{Product structure and continuous product structure} Consider 
\begin{align*}
\hSigma^+ &= \left\{\{R_n\}_{n \geq 0}: \exists\,\hy=\{S_n\}_{n\in\mathbb{Z}}\in \hSigma
 \text{ such that } \{R_n\}_{n \geq 0}=\{S_n\}_{n \geq 0}\right\}\\
\hSigma^- &= \left\{\{R_n\}_{n \leq 0}: \exists\,\hy=\{S_n\}_{n\in\mathbb{Z}}\in \hSigma
 \text{ such that } \{R_n\}_{n \leq 0}=\{S_n\}_{n \leq 0}\right\}.
\end{align*}
Points in $\hSigma^+$ will be denoted by $\hx^+$ and points in $\hSigma^-$ will be denoted by $\hx^-$. 
There are canonical projections $P^{+}: \hSigma \rightarrow \hSigma^{+}$ and $P^{-}: \hSigma \rightarrow \hSigma^{-}$, obtained by dropping all the negative coordinates respectively all of the positive coordinates
of elements of $\hSigma$. 

We define the \emph{local stable} and \emph{local unstable} sets of $\hx \in \hSigma$ by
\begin{align*}
W^{s}_{\rm loc}(\hx) &= W^{s}_{\rm loc}(P^+(\hx)) = \left\{\{S_n\}_{n \in\mathbb{Z}} \in \hSigma:R_n  = S_n \; \text{for all $n \geq 0$}\right\}\\
W^{u}_{\rm loc}(\hx) &= W^{u}_{\rm loc}(P^-(\hx)) = \left\{\{S_n\}_{n \in \mathbb{Z}} \in \hSigma:R_n  = S_n \; \text{for all $n \leq 0$}\right\}.
\end{align*}
We think of $\hSigma^{-}$ and $\hSigma^{+}$ as parametrizations of the local stable and unstable sets,
respectively. Observe that: if $\hy\in W^{s}_{\rm loc}(\hx)$ then
$d(\hsigma(\hx),\hsigma(\hy))\leq e^{-\frac{\chi}{2}}d(\hx,\hy)$;
if $\hy\in W^{u}_{\rm loc}(\hx)$ then $d(\hsigma^{-1}(\hx),\hsigma^{-1}(\hy))\leq e^{-\frac{\chi}{2}}d(\hx,\hy)$.

Given a symbol $R\in \cV$, define the \emph{cylinder} 
$$[R]=\left\{\{S_n\}_{n \in\mathbb{Z}} \in \hSigma: S_{0}  = R\right\}. $$

\begin{definition}\label{def: continuous prod struct}
A $\hsigma$-invariant probability measure $\hmu$ is said to have \emph{product structure} if the normalized restriction of $\hmu$ to every cylinder $[R]$ has the form $\hmu|_{[R]}=\hrho \times (\hnu^s\times \hnu^u)$ 
where $\hnu^s=P^-_*\hmu$ and $\hnu^u=P^-_*\hmu$ and the density $\hrho$ is measurable. Moreover, we say that $\hmu$ has \emph{continuous product structure} if the density $\hrho:\hSigma\to \real_+$ is uniformly continuous and bounded away from zero and infinity.
\end{definition}

\subsection{Invariant holonomies} A ($\alpha$-H\"older) \emph{stable holonomy} for the linear cocycle generated by $\hA:\hSigma \to GL(d,\field )$ over $\hsigma$ is a collection of linear maps $H^{s,\hA}_{\hx \hy} \in GL(d,\field )$ defined for $\hy \in W^{s}_{\rm loc}(\hx)$ which satisfy, for some $L>0$, the following properties:
\begin{enumerate}[$\circ$]
\item 
$H^{s,\hA}_{\hy \hz}=H^{s,\hA}_{\hx \hz}H^{s,\hA}_{\hy\hx} \text{ and } H^{s,\hA}_{\hx \hx} ={\rm Id}$;
\item 
$H^{s,\hA}_{\hsigma(\hy)\hsigma (\hz)}=\hA(\hz)H^{s,\hA}_{\hy \hz} \hA(\hy)^{-1};$
\item
$\left\|H^{s,\hA}_{\hy\hz}-{\rm Id}\right\| \leq L d(\hy,\hz)^{\alpha}$.
\end{enumerate}

Replacing the cocycle generated by $\hA$ over $\hsigma$ by the cocycle generated by $\hA$ over $\hsigma^{-1}$, we obtain an analogous definition for the \textit{unstable holonomies} $H^{u,\hA}_{\hx\hy}$, $\hy \in W^{u}_{\rm loc}(\hx)$.
Examples of linear cocycles admitting stable and unstable holonomies are given by \emph{locally constant cocycles} and \emph{$\frac{\chi}{2}$-fiber-bunched cocycles}. In these cases, the families of stable and unstable holonomies
are given by:
\begin{align*}
H^{s,\hA}_{\hx\hy}& = \lim _{n\rightarrow +\infty}\hA^n(\hy)^{-1}\hA^n(\hx), \; \; \hy \in W^{s}_{\rm loc}(\hx),\\
H^{u,\hA}_{\hx\hy}& = \lim _{n\rightarrow +\infty}\hA^{-n}(\hy)^{-1}\hA^{-n}(\hx), \; \; \hy \in W^{u}_{\rm loc}(\hx).
\end{align*}
See for instance \cite{BGV03, Almost}.

\section{Translation to the symbolic setting}
The goal of this section is to prove the proposition below, which provides a translation of our problem to the symbolic setting. Remember that $\mathcal B_\chi^{r,\alpha}(M)$ is the set of $C^{r,\alpha}$ cocycles that are
$\frac\chi{2}$-fiber-bunched over $(f,\mu)$.

\begin{proposition}\label{p.semiconj}
Let $f:M\to M$ be a $C^{1+\beta}$ diffeomorphism, and let $\mu$ be an ergodic $\chi$-hyperbolic
$u$-Gibbs measure. Given $A\in \mathcal B_\chi^{r,\alpha}(M)$,
there exists a topological Markov shift $(\hSigma,\hsigma)$ and a H\"older continuous map 
$\hpi:\hSigma\to M$ such that $\hpi\circ \hsigma = f\circ \hpi.$
Moreover:
\begin{enumerate}[{\rm (1)}]
\item There exists a $\hsigma$-invariant ergodic measure $\hmu$ with continuous product structure such that
$\hpi_{\ast} \hmu=\mu$.
\item The map $\hA:\hSigma\to \GL$ defined by $\hA:=A\circ \hpi$ is H\"older continuous and admits invariant holonomies.
\end{enumerate}
\end{proposition}

The existence of the semi-conjugacy $\hpi:\hSigma\to M$ follows readily from our assumptions and Theorem \ref{theo: Markov partitions}, so it remains to prove items (1) and (2).

We begin with (2).
Let $\hA:\hSigma\to \GL$ be the map given by $\hA(\hx)=A( \hpi(\hx))$. Observe that, since $\hpi$ is Lipschitz (see Remark \ref{rem: pi is lipschitz}), $\hA$ is $\alpha$-H\"older. Moreover, by the definition of $\hA$ and
the assumptions on $A$,
\begin{displaymath}
\| \hA^n(\hx)\| \| \hA^{n}(\hx)^{-1}\| (e^{-\frac{\chi}{2}})^{| n| \alpha}\leq \theta ^{| n|},
\ \forall \hx\in \hSigma, \forall n\in \mathbb{Z}.
\end{displaymath} 
Since $d(\hsigma(\hx),\hsigma(\hy))\leq e^{-\frac{\chi}{2}}d(\hx,\hy)$ for $\hx,\hy\in\hSigma$ in the same stable manifold (similarly for points in the same unstable manifold under iteration of $\hsigma^{-1}$), it follows that $\hA$ is fiber-bunched in the sense of \cite{AvV1}.
By \cite[Proposition A.6]{AvV1}, $\hA$ admits stable and unstable. This proves (2). 

To prove (1), we proceed as in \cite{Sa13}: since the semi-conjugacy $\hpi$ satisfies property (3)
of Theorem \ref{theo: Markov partitions}, we can define 
the measure $\hmu$ by 
\begin{equation}\label{eq: def extension of mu}
\hmu(E):=  \displaystyle\int _M \left(\frac{1}{\abs{\hpi^{-1}_\#(x)}}\sum_{\hpi_{\#}(\hx)=x}\car_{E}(\hx)\right)d\mu(x)
\end{equation}
where $\hpi_\#:\hSigma^\#\to M$ is the restriction of $\hpi$ to the recurrent set $\hSigma^\#$.
The fact that $\hmu$ is indeed a probability measure is proved in \cite[Proposition~13.2]{Sa13}.
The same proposition shows that $\hmu$ is $\hsigma$-invariant and satisfies $\hpi_*\hmu=\mu$.
Furthermore, almost every ergodic component of $\hmu$ also projects to $\mu$. Passing to an
ergodic component, it remains to prove the continuous product structure.

\subsection{Extensions of Gibbs measures have continuous product structure }

A measurable set $Q\subset M$ is said to have \emph{local product structure} if for every $x,y\in Q$ there exist Pesin stable and unstable manifolds $W^s_{\rm loc}(x),W^u_{\rm loc}(y)$ that intersect in a unique point,
and this intersection point belongs to $Q$. We denote the intersection point by $[x,y]$. In this case, for any $x_0\in Q$
there are measurable sets $\cN^u_Q\subset W^u_{\rm loc}(x_0)$, $\cN^s_Q\subset W^s_{\rm loc}(x_0)$
and a bi-measurable map from $\Phi:\cN^u_Q\times \cN^s_Q\to Q$ defined by $\Phi(x,y)\mapsto [x,y]$.

 \begin{definition}\label{def:marginalsLPS}
A hyperbolic measure $\mu$ is said to have \emph{local product structure} if, for every measurable set $Q\subset M$ with local product structure and $\mu(Q)>0$, the pull back measure $\nu=\Phi^{-1}_*\mu|_Q$ of $\mu $ restricted to $Q$ is equivalent to $\nu^u\times \nu^s$, where $\nu^u$ and $\nu^s$ are the projections of $\nu$ to $\cN^u_Q$ and $\cN^s_Q$, respectively.
\end{definition}

We stress that the notions of product structure and continuous product structure stated by Definition \ref{def: continuous prod struct} and the notion of local product structure defined above do not coincide, in general.

\begin{remark}
Observe that by the absolute continuity of stable and unstable holonomies any Gibbs state ($u$-Gibbs or $s$-Gibbs) has local product structure.
\end{remark}

In the remaining of this section we prove that the measure $\hmu$ defined by \eqref{eq: def extension of mu} has continuous product structure. We begin with some auxiliary results about $\hpi$ given by Theorem
\ref{theo: Markov partitions}, as detailed in Appendix \ref{appendix}.

\begin{lemma}\label{l.product}
For every $R\in \cV$, $\hpi([R])\subset M$ is a compact subset with local product structure. Moreover,
$\hpi|_{[R]}:[R]\to \hpi([R])$ preserves $[\cdot,\cdot]$, that is:
$$\hpi([\hx,\hy])=[\hpi(\hx),\hpi(\hy)],\ \forall \hx,\hy \in[R].$$
\end{lemma}

The proof of Lemma \ref{l.product} is based on \cite{Sa13}, and is contained in Appendix \ref{appendix}.

%

\begin{remark}\label{r.product}
Observe that $\hSigma^\#$ is closed by $[\cdot,\cdot]$, but $\hSigma^\#$ is usually not compact. 
By the Poincar\'e recurrence theorem, $\hSigma^\#\subset \hSigma$ has full $\hmu$-measure,
and by part (2) of Theorem \ref{theo: Markov partitions} the set
$\hpi[\hSigma^\#]$ has full $\mu$-measure.
A final observation is that, by Lemma~\ref{l.product} above, $\hpi([R]\cap\hSigma^\#)$ has product structure but it is not necessarily compact. 
\end{remark}

Given a cylinder $[R]\subset \hSigma$, let $\hmu_R:=\hmu|_{[R]}$. By Rokhlin's disintegration theorem (see \cite{FET}), $\hmu _R$ has a disintegration along local stable sets $W^s_{\rm loc}(\hx^+)$, say
$$
\hmu_R=\int_{\hSigma^+} \hmu^s_{\hx^+} d\hmu^u(\hx^+).
$$
Similarly, since $\widetilde{R}=\hpi([R])\subset M$  has local product structure by Lemma~\ref{l.product}, then there exists $\Phi:\cN_{\widetilde{R}}^u\times \cN_{\widetilde{R}}^s\to \widetilde{R}$ (see definition~\ref{def:marginalsLPS}).
From now on lets work in the coordinates $\cN^u\times \cN^s:=\cN_{\widetilde{R}}^u\times \cN_{\widetilde{R}}^s$ given by $\Phi$, $\mu_{\widetilde{R}}:=\Phi^{-1}_*\mu|_{\widetilde{R}}$ has a
disintegration along stable manifolds, say
$$
\mu_{\widetilde{R}}=\int_{\cN^u} \mu^s_{x^u} d\mu^u(x^u).
$$

The next lemma gives us a relation between the disintegrations of $\hmu_R$ and $\mu_{\widetilde{R}} $. 
To state it precisely, let us introduce the following notation: given $\hx\in \hSigma$, write $\hx=(\hx^+,\hx^-)$ where $\hx^- =P^-(\hx)$ and $\hx^+ =P^+(\hx)$; given $x=\hpi(\hx)\in M$, write $x=(x^u,x^s)$ where
$(x^u,x^s)\in \cN^u\times\cN^s$ is the inverse image of $x$ by $\Phi$. Although these notations do depend on the choices of $\cN^s$ and $\cN^u$,
they are locally well-defined, so we abuse notation and write
$\hpi ^-(\hx^-)=x^s$ and $\hpi^+(\hx^+)=x^u$. By Remark~\ref{r.product},
we can also write $\hpi_\#^-(\hx^-)=x^s$ and $\hpi_\#^+(\hx^+)=x^u$ for $\underline R\in\hSigma^\#$.

\begin{lemma}\label{l.disintegration}
 For $P^+_*\hmu$-almost every $\hx^+$ and every measurable set $B_{\hx^+}\subset \lbrace \hx^+\rbrace \times \hSigma^-$, it holds
 $$
 \hmu^s_{\hx^+}(B_{\hx^+})=\int_{\cN^s}\frac{1}{\abs{\hpi_\#^{-1}(x^u,x^s)}}\sum_{\hpi_\#(\hx^+,\hx^-)=(x^u,x^s)}\car_{B_{\hx^+}}(\hx^-)d\mu^s_{\hpi ^+(\hx^+)}(x^s).
 $$
\end{lemma}

\begin{proof}

Given a symbol $R\in \cV$, let $\tR:=\hpi([R])$. Fix $\widehat{p}\in [R]$, and let $p=\hpi(\widehat{p})$. Let $\cN^u\times \cN^s$ be coordinates on $\tR$ centered at $p$, induced by the local product structure. By Lemma~\ref{l.product}, the map $\hpi|_{[R]} :[R] \to \tR$ preserves $[\cdot,\cdot]$. Using the product structure coordinates $\hSigma^+\times \hSigma^-$ on $[R]$, it follows that $\hpi|_{[R]} :[R] \to \cN^u\times \cN^s$
is given by
$$
\hpi(\hx^+,\hx^-)=(\hpi^+(\hx^+),\hpi^-(\hx^-))=(x^u,x^s).
$$
Let $B \subset [R]$ be a measurable set with $\hmu(B)>0$. We have
$$
 \hmu_R(B)=\int_M\left(\frac{1}{\abs{\hpi_\#^{-1}(x^u,x^s)}}\sum_{\hpi_\#^+(\hx^+)=x^u}\sum_{\hpi_\#^-(\hx^-)=x^s}\car_{B_{\hx^+}}(\hx^-)\right)d \mu_{\tR},
$$
where $B_{\hx^+}=\lbrace \hy^-\in \hSigma^-:(\hx^+,\hy^-)\in B\rbrace$. Consequently,
$$
\hmu_R(B)= \int_{\cN^u} \sum_{\hpi_\#^+(\hx^+)=x^u}\hmu^s_{\hx^+}(B_{\hx^+})d\mu^u(x^u),
$$
where
$$
\hmu^s_{\hx^+}(B_{\hx^+})=\int_{\cN^s}\frac{1}{\abs{\hpi_\#^{-1}(x^u,x^s)}} \sum_{\hpi_\#^-(\hx^-)=x^s}\car_{B_{\hx^+}}(\hx^-)d\mu^s_{x^u}(x^s).
$$
Since disintegrations are uniquely defined almost everywhere, the result follows. 
\end{proof}

A simple consequence of the lemma above is the following corollary.

\begin{corollary}\label{c.muproduct} 
If $\mu$ has local product structure then $\hmu$ has product structure. Moreover, if $\mu$ has local product structure with density $\rho$ then $\hmu$ has product structure with density $\hrho$ given by $\hrho(\hx)=\rho(\hpi(\hx))$.
\end{corollary}

\begin{proof}
The first part of the claim is a direct consequence of Lemma \ref{l.disintegration}. To prove the second part, 
write $\Phi^{-1}_*\mu|_{\hpi([R])}=\rho \nu$ where $\nu:=\nu^s\times\nu^u$ and $\nu^s$, $\nu^u$ are 
as in Definition~\ref{def:marginalsLPS}. Let $\hnu$ be the lift of $\nu$ as in ~\eqref{eq: def extension of mu}.
By definition, if $\widehat{F}\subset \hSigma$ is a measurable set and
$\varphi:\hSigma\to \real$ is a measurable function, then 
\begin{align*}
\hnu(\hF)&=\int_{M} \frac{1}{\abs{\hpi_\#^{-1}(x)}} \left(\sum_{\hpi_\#(\hx)=x}\car_{\hF}(\hx)\right)d\nu(x)\\
\int \varphi d\hnu&=\int_{M} \frac{1}{\abs{\hpi_\#^{-1}(x)}} \left(\sum_{\hpi_\#(\hx)=x}\varphi(\hx)\right)d\nu(x).
\end{align*}
Hence
$$
\int_{\hSigma}\car_{\hF}(\hx)\rho(\hpi(\hx))d\hnu=\int_{M} \frac{1}{\abs{\hpi_\#^{-1}(x)}} \left(\sum_{\hpi_\#(\hx)=x}\car_{\hF}(\hx)\rho(x)\right)d\nu(x)
$$
and we conclude that 
$$
\hmu(\widehat{F})=\int_M \frac{1}{\abs{\hpi_\#^{-1}(x)}} \left(\sum_{\hpi_\#(\hx)=x}\car_{\hF}(\hx)\right)\rho(x)\, d\nu(x)=\int_{\hF} \hrho d\hnu.
$$

\end{proof}

Now we are ready to prove that $\hmu$ has continuous product structure. 
Assume that $\mu$ is an $s$-Gibbs state (the case of $u$-Gibbs states is analogous).

\begin{proposition}\label{p.gibbsproduct}
If $\mu$ is an $s$-Gibbs state then $\hmu$ has product structure with density function $\hrho:\hSigma\to \real_+$ uniformly continuous and bounded away from zero and infinity.
\end{proposition}

In Corollary \ref{c.muproduct} we proved that $\hmu$ has product structure and its density function
$\hrho$ is given by $\hrho(\hx)=\rho (\hpi(\hx))$. It remains to show that 
$\hrho$ is uniformly continuous and bounded away from zero and infinity.
By the assumption that $\mu$ is an $s$-Gibbs state, it admits a disintegration
$(\mu^s_p)_p$ on stable leaves such that $\mu^s_p$ is absolutely continuous with
respect to the Lebesgue measure on the respective stable leave for almost every $p$. 
In other words, if $m^s_p$ is the Lebesgue measure in the stable manifold of $p\in M$,
then $\mu^s_p=\psi \, m^s_p$ for some density function $\psi$. By \cite{Le84a},
if we denote by $JDf^n|_{E^s(x)} $ the Jacobian of $Df^n$ along $E^s_x$, then
$\psi$ satisfies
\begin{equation}\label{eq.density}
\Delta(x,y):=\frac{\psi(x)}{\psi(y)}=\lim_{n\to \infty}\frac{JDf^n|_{E^s(x)} }{JDf^n|_{E^s(y)} }
\end{equation}
for every $x,y$ in the same stable manifold.
Let $\Gamma\subset \hSigma\times \hSigma$ be the set defined by $\Gamma=\lbrace (\hx,\hy)\in \hSigma\times \hSigma:R_n=S_n,\forall n\geq 0\rbrace$, and consider the function $\hDelta:\Gamma\to \real$ given by $\hDelta(\hx,\hy)=\Delta(\hpi(\hx),\hpi(\hy))$. We now study the regularity properties of $\hDelta$.
Let us recall the notion of weak H\"older continuity.

\begin{definition}
Let $(X,\d _X)$ be a metric space. A function $g:X \to \real$ is called \emph{weak H\"older} if there exists $\gamma>0$ such that for every $\epsilon>0$ there exists $C(\epsilon)>0$ such that 
$$
\abs{g(x)-g(y)}\leq C(\epsilon)\d_X(x,y)^\gamma + \epsilon,\ \forall x,y\in X.
$$
\end{definition}

Clearly, every weak H\"older functions is absolutely continuous.

\begin{lemma}\label{l.Delta}
The function $\hDelta$ is weak H\"older (hence uniformly continuous), and it is bounded away from zero and infinity.
\end{lemma}
\begin{proof}
Define $Jf^s(\hx):=JDf|_{E^s(\hpi(\hx))}$ and let 
 $$ \hDelta_n(\hx,\hy)=\frac{JDf^n|_{E^s(\hpi(\hx))} }{JDf^n|_{E^s(\hpi(\hy))} },\ (\hx,\hy)\in\Gamma.$$
The map $\hx\to E^s(\hpi(\hx))$ is H\"older continuous (the two dimensional case has been proved by Sarig, the necessary tools to prove the multidimensional version of this lemma have been developed in \cite{Ova16}, hence the same proof of the case of surfaces establishes the higher dimensional case). Since $f$ is $C^{1+\beta}$
and $M$ is compact, $\hx\mapsto \log Jf^s(\hx )$ is H\"older continuous and bounded away from zero. Let $\gamma$ be its H\"older exponent and $C>0$ the corresponding constant. Thus, for every $(\hx,\hy)\in \Gamma$,
\begin{displaymath}
\begin{split}
\left| \log \hDelta_n (\hx,\hy) \right|&= \left| \sum_{i=0}^{n-1} \log Jf^s[\hsigma^i(\hx)]-\log Jf^s[(\hsigma^i(\hy)]\right| \\
&\leq \sum_{i=0}^{n-1} \left| \log Jf^s[\hsigma^i(\hx)]-\log Jf^s[(\hsigma^i(\hy)]\right| \\
&\leq C \sum_{i=0}^{n-1}d(\hsigma^i(\hx),\hsigma^i(\hy))^{\gamma}\leq C \frac{1}{1-e^{-\chi \gamma/2}}d(\hx,\hy)^{\gamma}. 
\end{split}
\end{displaymath} 
Therefore, $\log \hDelta_n$ converges uniformly to $\log \hDelta$ and 
\begin{equation}\label{eq: bounds on gamma}
   e^{-C/(1-e^{-\chi \gamma/2})}\leq \hDelta(\hx,\hy)\leq e^{C/(1-e^{-\chi \gamma/2})}.
\end{equation}
We now claim that $\log \hDelta$ is weak H\"older. Indeed, if $(\hx,\hy),(\hz,\hp)\in\Gamma$ then
\begin{align*}
 &\log \hDelta_n(\hx,\hy)-\log \hDelta_n(\hz,\hp)\\
 &= \sum_{i=0}^{n-1} \left(\log Jf^s[(\hsigma^i(\hx)]-\log Jf^s[(\hsigma^i(\hy)]-\log Jf^s[(\hsigma^i(\hz)]
 +\log Jf^s[(\hsigma^i(\hp)]\right). 
  \end{align*}
Recalling that $\hx\mapsto \log Jf^s(\hx)$ is $(C,\gamma)$-H\"older, we obtain that
\begin{align*}
&| \log \hDelta_n(\hx,\hy)-\log \hDelta_n(\hz,\hp)|\\
&\leq C   \min_{1\leq j\leq n} \bigg( \sum_{i=0}^{j-1}d(\hsigma^i(\hx),\hsigma^i(\hz))^\gamma+d(\hsigma^i(\hy),\hsigma^i(\hp))^\gamma+\\
&\hspace{2cm}\sum_{i=j}^{n-1}d(\hsigma^i(\hx),\hsigma^i(\hy))^\gamma+d(\hsigma^i(\hz),\hsigma^i(\hp))^\gamma  \bigg),
\end{align*} 
where the first summand in the latter expression is related to distances of points
that are not necessarily in the same stable manifold and the second summand to distances of points
in the same stable manifold.
Now, since $\hsigma$ is Lipschitz with Lipschitz constant $L=e^{\chi/2}$, it follows that
\begin{align*}
&\left|\log \hDelta(\hx,\hy)-\log \hDelta(\hz,\hp)\right|\\
&\leq C  \min_{n\in \natural} \left( \frac{L^{n\gamma}-1}{L^\gamma-1}\left(d(\hx,\hz)^\gamma+d(\hy,\hp)^\gamma\right)+\frac{2e^{-\chi \gamma n/2}}{1-e^{-\chi\gamma/2}}\right).
\end{align*}
Given $\epsilon>0$ if we take $n$ large enough such that
$2C\frac{e^{-\chi \gamma n/2}}{1-e^{-\chi\gamma/2}}\leq \epsilon$ and $C(\epsilon)=2 C \frac{L^{\gamma n}-1}{L^\gamma-1}$, then the claim follows. Consequently, using \eqref{eq: bounds on gamma} and observing that $\hDelta=e^{\log \hDelta}$ we complete the proof of the proposition.
\end{proof}

Fix a cylinder $[R]\subset \hSigma$ and let $\hmu^s_{\hx^+}$ be the measure on $\{\hx^+\} \times \hSigma^-$ given by Lemma~\ref{l.disintegration}. Recalling that $\mu^s_p=\psi \, m^s_p$,
take $\hpsi(\hx)=\psi[\hpi(\hx)]$. We normalize the map $\hpsi$ so that
$\displaystyle\int_{\hSigma^-}\hpsi(\hx)d \hm^s_{\hx^+}=1$. By \eqref{eq.density} it follows that
$$
\hpsi(\hy)= \frac{1}{ \displaystyle\int_{\hSigma^-}\hDelta(\hy,(\hx^-,\hy^+))d\hm_{\hx^+}^s(\hx^-)},\ \forall\hy \in \hSigma.
$$
In particular, Lemma \ref{l.Delta} implies that $\hpsi:[R]\to \real_+$ is uniformly continuous and bounded away from zero and infinity.

Let $\widetilde{R}:=\hpi([R])\subset M$. By Lemma \ref{l.product}, $\widetilde{R}$ has local product structure.
Working again in the coordinates $\cN^u\times\cN^s$, Rokhlin's Disintegration Theorem gives that 
$$\mu|_{\widetilde{R}}=\int_{\cN^u} \psi(x) m^s_{x^u}\times \nu^u $$
where $m^s_{x^u}$ is the volume measure in the manifold $W^s(x^u):=\lbrace x^u \rbrace \times \cN^s$. 
Moreover, in the product structure coordinates the unstable holonomy $h^u_{x^u,y^u}:W^s(x^u)\to W^s(y^u)$
is a map
$h^u_{x^u,y^u} \colon \lbrace x^u \rbrace \times \cN^s\to\lbrace y^u \rbrace \times \cN^s$ given by 
$h^u_{x^u,y^u}(x^u, x^s)=(y^u,x^s)$.
By the absolute continuity of the unstable foliation $(h^u_{x^u,y^u})_*m^s_{x^u}=Jh^u_{x^u,y^u} m^s_{y^u}$,
where $Jh^u_{x^u,y^u}$ is given by (see \cite{PSh89}): 
$$
Jh^u_{x^u,y^u}(x^s)=\lim_{n\to \infty} \frac{JDf^{-n}|_{E^s(x^u,x^s)}}{JDf^{-n}|_{E^s(y^u,x^s)}}\cdot
$$
We now fix some $(y^u,y^s)\in \cN^u\times \cN^s$ and define 
$\phi:\tR\to\real$ by $\phi(x^u,x^s)=Jh^u_{x^u,y^u}(x^s)$.
Then $\mu_{\tR}=\psi\phi\times(m^s_{y^u}\times \nu^u)$. Moreover, considering
$\hphi:[R]\to \real$, $\hphi(\hx):=\phi[\hpi(\hx)]$, 
where $\hpi(\hx)=\hpi(\hx^+,\hx^-)=(x^u,x^s)$, an analogous calculation to the one did in the
proof of Lemma~\ref{l.Delta} shows that $\hphi$ is uniformly continuous and bounded away from zero and infinity.
By Corollary \ref{c.muproduct} and the equality $\rho=\psi\phi$, we obtain
that $\hrho=\hpsi\hphi$ and so $\hrho$ is also uniformly continuous and bounded away from zero and infinity.
This concludes the proof of Proposition~\ref{p.gibbsproduct}. 

To conclude the proof of Proposition~\ref{p.semiconj}, it remains to show that
$\hmu$ can be chosen to be ergodic. This follows from the next lemma.

\begin{lemma}
Almost every ergodic component of $\hmu$ projects to $\mu$ and has continuous product structure.
\end{lemma}
\begin{proof}
 By \cite[Proposition~13.2]{Sa13} almost every ergodic component of $\hmu$ projects to $\mu$. We claim that the ergodic decomposition of $\hmu$ is actually a sum of restrictions of $\hmu$ to union of cylinders. Indeed, let $\hSigma_0\subset \hSigma$ be a full $\hmu$-measure subset on which the Birkhoff theorem 
holds for every continuous function. Divide $\hSigma_0$ into equivalence classes: $\hx \sim \hy$ if the Birkhoff averages are the same for every continuous function. For any cylinder $[R]$, using a Hopf argument we conclude that $\hmu$-almost every point in $[R]$ is in the same equivalence class. So, every equivalence class is a union of cylinders, modulo sets of zero measure. In particular, there are at most countably many classes, which we denote by $\Gamma_j$, $j\in \natural$. Hence $\hmu=\sum_{j}\hmu_j$ where
every $\hmu_j:=\frac{1}{\hmu(\Gamma_j)}\hmu|_{\Gamma_j}$ is ergodic. Moreover, the restriction of $\hmu_j$ to a cylinder 
of positive measure is a multiple of the restriction of $\hmu$ to this cylinder. In particular, every $\hmu_j$ has continuous product structure.
\end{proof}

The proof of Proposition \ref{p.semiconj} is now complete.

\section{Simplicity is typical}

In this section we conclude the proof of Theorem \ref{theo: main}. We begin recalling a criterion established by Avila and Viana in \cite{AvV1} that guarantees the simplicity of the Lyapunov spectrum of a fiber-bunched cocycle over
a topological Markov shift. 
 
\subsection{Simplicity Criterion} Let $\hp\in\hSigma$ be a $\hsigma$-periodic point with period $q\geq 1$.
A point $\hz\in W^u_{\rm loc}(\hp)$ is called \textit{homoclinic} if there exists some multiple $\ell\geq 1$ of $q$ such that $\hsigma^\ell(\hz)\in W^s_{\rm loc}(\hp)$. The \textit{transition map} $\psi^{\hA}_{\hp,\hz}:\field^d\to \field^d$ is defined by
$$
 \psi^{\hA}_{\hp,\hz}=H^{s,\hA}_{\hsigma^l(\hz)\hp}\circ \hA^\ell(\hz)\circ H^{u,\hA}_{\hp\hz}.
$$
\begin{definition}\label{d:simple}
A cocycle $\hA: \hSigma \to GL(d,\mathbb K)$ is \emph{simple} if there exists a $\hsigma$-periodic point $\hp\in \hSigma$ of period $q\ge 1$ and some homoclinic point $\hz \in W_{\rm loc}^u(\hat p)$ such that:
\begin{enumerate}[i)]
 \item[(P)] all eigenvalues of $\hA^q(\hp)$ have distinct absolute values;
 \item[(T)] for any invariant subspaces (sums of eigenspaces) $E$ and $F$ of $\hA^q(\hp)$ with $\dim E+\dim F=d$,  it holds $\psi^{\hA}_{\hp,\hz}(E)\cap F=\lbrace 0\rbrace$.
\end{enumerate}
\end{definition}

Property (P) is called \textit{pinching} and property (T) is called \textit{twisting}.
It was proved in \cite[Theorem~A]{AvV1} that pinching and twisting imply that the
Lyapunov spectrum is simple. More precisely,

\begin{theorem}[Theorem A of \cite{AvV1}]\label{thm:AV}
If $\hA: \hSigma \to GL(d,\mathbb K)$ is simple then the cocycle generated by $\hA$ over $\hsigma$ has simple Lyapunov spectrum.
\end{theorem}

Observe that, in order to apply Theorem \ref{thm:AV}, we only need $(\hSigma,\hsigma)$ to be a topological Markov shift and the measure $\hmu$ to have continuous product structure in the sense of Definition \ref{def: continuous prod struct}. This fact follows from \cite[Appendix A.1]{AvV1}.
In our context, Proposition \ref{p.semiconj} establishes that this is exactly the case for the topological
Markov shift $(\hsigma, \hmu)$ and measure $\hmu$ induced by the pair $(f,\mu)$ satisfying the
hypothesis of Theorem \ref{theo: main}.
 
\subsection{Conclusion of the proof of Theorem \ref{theo: main}} 
Keeping all the notation introduced in the previous sections, we will prove that
$\hA$ satisfies Definition \ref{d:simple}.
Before going into the proof, let us remind part (3) of Theorem \ref{theo: Markov partitions}:
there exists a function $\varphi:\mathfs R\to\mathbb N$ such that if $x=\hpi(\hx)$ with $R_n=R$ for infinitely many
$n>0$ and $R_n=S$ for infinitely many $n<0$ then $\#\{\hy\in \hSigma^\#:\hpi(\hy)=x\}\leq \varphi(R)\varphi(S)$.



\begin{proposition}\label{p.pintwistdense}
Let $f:M\to M$ and $A:M\to GL(d,\mathbb K)$ be as in Theorem \ref{theo: main}. Then, there exists $B:M\to GL(d,\mathbb K)$ $C^{r,\alpha}$-arbitrarily close to $A$ such that $\hB=B\circ \hpi$ is simple.
\end{proposition}

\begin{proof}
Assume first that $\field =\mathbb{C}$. Let $\hp \in \hSigma$ be a $\hsigma$-periodic point with period $q$, let $\hz \in W_{\rm loc}^u(\hp)$ be a homoclinic point such that $\hsigma^\ell(\hz)\in W^s_{\rm loc}(\hp)$ and consider $p=\hpi(\hp)$ and $z=\hpi(\hz)$. Since $A$ takes values in $\GLC$,
we can perform a small $C^{r,\alpha}$-perturbation of $A$ on a neighborhood of $p$ such that
the perturbed map $B':M\to \GLC$ is arbitrarily close to $A$ and all eigenvalues of $B'^q(p)$ have distinct absolute values. In particular, $\widehat{B'}=B'\circ \hpi $ satisfies the pinching condition.
 
To obtain the twisting property, we begin observing that $\hpi[\hsigma^{nq}(\hz)]\neq \hpi[\hsigma^{jq}(\hz)]$ for every pair of distinct integers $n,j$. Indeed, if this is not the case then there are $n\neq j$ such that
$\hpi[\hsigma^{nq}(\hz)]=\hpi[\hsigma^{jq}(\hz)]$; using that $\hpi \circ \hsigma^\ell = f^\ell \circ \hpi$ for all
$\ell\in\mathbb Z$, we get that $f^{q(j-n)}(z)=z$ and so is $z$ is periodic, which contradicts
the choice of $\hz$. Thus, there exists a small neighborhood $V\subset M$ of $z$ such that
$\hz\notin\hsigma^{nq}[\hpi^{-1}(V)]$ for every $n\neq 0$. In particular, modifying $B'$ in this neighborhood $V$ does not change neither of the holonomies $H^{s,\hB'}_{\hsigma^\ell(\hz)\hp}$ and $H^{u,\hB'}_{\hp\hz}$.
In particular, 
$\psi^{\hB}_{\hp,\hz}=H^{s,\hB'}_{\hsigma^\ell(\hz)\hp}\circ {B}^\ell(z)\circ H^{u,\hB'}_{\hp\hz}$ for any $C^{r,\alpha}$-perturbation 
$B:M\to \GLC$ of $B'$ that is supported on $V$. Consequently, there is a cocycle $B$ that is $C^{r,\alpha}$-arbitrarily close to $B'$ , coincides with $B'$ outside $V$, and such that $\psi^{\hB}_{\hp\hz}$ does not preserve the invariant subspaces of $\hB^q(\hp)$, hence $\hB$ has the twisting property. Noting that, since $B$ and $B'$ coincide
outside $V$, $\hB$ still satisfies the pinching condition. This concludes the proof of the proposition
when $\field =\mathbb{C}$. 

Now assume that $\field =\mathbb{R}$. The first perturbation performed in the case $\field=\mathbb C$
can also be performed when $\field =\mathbb{R}$, therefore we can assume that 
$\hA$ already satisfies the twisting property at a periodic point $\hp$ of period $q\ge 1$.
The difficulty is to obtain the pinching property, because there may exist pairs of complex eigenvalues.
To bypass this issue, we explain how to adapt ideas from \cite[Section~9]{BoV04} to our context.

After a small perturbation, if necessary, we can assume that there exists a splitting 
$\mathbb R^d=E^1(\hp)\oplus \cdots \oplus E^k(\hp)$ into invariant subspaces of $A^{q}(\hpi(\hp))$,
where each $E^j(\hp)$ is a one or two dimensional eigenspace and the eigenvalues corresponding 
to different subspaces have different absolute values. This perturbation can be done in a way that 
$\hA$ still satisfies the twisting condition at $\hp$. 
If all subspaces $E^j(\hp)$ are one dimensional, then we are done. So, let us assume that $\dim E^j(\hp)=2$
for some $j$, that is $E^j(\hp)$ is associated to a complex eigenvalue of $A^{q}(\hpi(\hp))$.
Since the cocycle $\hA$ admits stable and unstable holonomies and has the twisting property
at $\hp$, there exists a horseshoe $H$ containing $\hp$ and $\hz$ and a dominated decomposition $E^1\oplus \cdots \oplus E^k$ over $H$ that extends $E^1(\hp)\oplus \cdots \oplus E^k(\hp)$, see \cite[Section~9]{BoV04}. 
For $t\in [0,1]$, let $R_{t\delta }: M \to GL(d,\mathbb R)$ be a $C^\infty$ cocycle on $M$ 
such that the matrix $R_{t\delta}(p)$ is a rotation of angle $t\delta$ when restricted to the plane
$E^j(\hp) \subset \mathbb R^d$, and is the identity map  when restricted to the other subspaces
$E^i(\hp)$, $i\neq j$. Consider the continuous
family of cocycles
$A_{\delta ,t} :=R_{t\delta }A \in \mathcal B_\chi^{r,\alpha}(M)$. 

By the symbolic dynamics, there exists a sequence   $(\hx_n)_n$ of periodic points of $H$ such that each $\hx_n$ has period $nq+\ell$, the points $\hsigma^{i}(\hx_n)$ and $\hsigma^{i}(\hz)$ are close for every $0\leq i\leq \ell$,
and the points $\hsigma^{\ell+i}(\hx_n)$ and $\hsigma^i(\hp)$ are close for every $0\leq i\leq qn$. 

As we will perturb the cocycle generated by $A:M\to GL(d,\mathbb{R}) $ over $f$, we need real eigenvalues for the perturbation of $A^{\per(x_n)}(x_n)$, where $x_n:=\hpi(\hx_n)$. The point $x_n$ is clearly periodic for $f$. If the projection $\hpi$ restricted to the orbit  $\lbrace\hsigma^j(\hx_n),\,j\geq 1\rbrace$ is injective, then the argument given in \cite[Section~9]{BoV04} works directly in our context. If not, we need to estimate the period of $x_n$. 

If $u\in \cV$ denotes the symbol such that $\hp\in [u]$ then the symbol $u$ appears infinitely many times
in the coding of the periodic point $\hx_n$ and, by part (3) of Theorem~\ref{theo: Markov partitions}, $\hpi^{-1}(x_n)$ has cardinality less than or equal to $m:=\varphi(u)^2<\infty$. Therefore, the period of $x_n$ satisfies
$\per(x_n)\geq \frac{nq+\ell}{m}$. Now, the argument of \cite[Section~9]{BoV04} shows that the variation of the rotation number of $A^{\per(x_n)}_{\delta,t}(x_n)$ is at least $\frac{nt\delta}{2 m}$. Thus, for $n$ sufficiently large we can find $t$ close to $0$ so that $A^{\per(x_n)}_{\delta,t}(x_n)$ has a real eigenvalue of multiplicity $2$ in the plane $E^j(\hx_n)$. Then, making an extra small 
$C^{r,\alpha}$-perturbation near the point $x_n$, we obtain two different real eigenvalues on $E^j(\hx_n)$.

Repeating this process a finite number of times (in fact, no more than $d$ times)
we find a cocycle $B$ close to $A$ and a periodic point $\hp\in \hSigma$ that has both
the twisting and pinching properties.
\end{proof}

Proposition \ref{p.pintwistdense} implies that the set of cocycles $B\in \mathcal B_\chi^{r,\alpha}(M)$   
for which $\hB$ has the pinching and twisting conditions is dense in $\mathcal B_\chi^{r,\alpha}(M)$. We now observe that this set is also open. Consider the map
$\hpi^*: \mathcal B_\chi^{r,\alpha}(M) \to H^{\alpha}(\hSigma)$ given by $\hpi^*(A)=\hA=A\circ \hpi$.
This map is Lipschitz continuous, since the $C^{0,\alpha}$-norm of $\hpi^* A$ on $\hSigma$
is bounded by
\begin{align*}
&\norm{\hpi^* A}_{0,\alpha}=\sup_{\hx\in\hSigma}{\norm{\hA(\hx)}}+\sup_{\hx,\hy\in\hSigma\atop{\hx\neq\hy}}\frac{\norm{\hA(\hx)-\hA(\hy)}}{d(\hx,\hy)^\alpha}\\
&=\sup_{x\in M}\norm{A(x)}+\sup_{\hx,\hy\in\hSigma\atop{\hpi(\hx)\neq\hpi(\hy)}}
\left(\frac{\norm{A(\hpi(\hx))-A(\hpi(\hy))}}{\d(\hpi(\hx),\hpi(\hy))^\alpha}\right)\left(\frac{\d(\hpi(\hx),\hpi(\hy))}{d(\hx,\hy)}\right)^\alpha\\
&\leq \max\lbrace{1,\text{Lip}(\hpi)^{\alpha}\rbrace} \norm{A}_{0,\alpha}.
\end{align*}
Since the pinching and twisting conditions are open conditions in $H^{\alpha}(\hSigma)$, given a cocycle 
$B\in \mathcal B_\chi^{r,\alpha}(M)$ such that $\hpi^*B$ satisfies these two conditions it follows that if 
$B'\in \mathcal B_\chi^{r,\alpha}(M)$ is $C^{r,\alpha}$ sufficiently close to $B$, then $\hpi^*B'$ also satisfies
the pinching and twisting conditions, because $\hpi^*B'$ is close to $\hpi ^{\ast}B$. This proves our claim.

Since $(A,f,\mu)$ and $(\hA, \hsigma ,\hmu)$ have the same Lyapunov spectrum, an application
of Theorem \ref{thm:AV} implies that $\mathcal B_\chi^{r,\alpha}(M)$ contains
an open and dense subset of cocycles with simple spectrum, thus
concluding the proof of the Theorem \ref{theo: main}.

\appendix
\section{Symbolic dynamics for non-uniformly hyperbolic diffeomorphisms (by Yuri Lima)} \label{appendix}

Here we give, as much as possible, a brief sketch of the proof of Theorem \ref{theo: Markov partitions},
as well as a self contained proof of Lemma \ref{l.product}, based on the work
of Sarig \cite{Sa13} and of Ben Ovadia \cite{Ova16}. We eventually refer the reader to \cite{LS16} and \cite{LM16},
where some of the arguments
have been simplified and/or are better explained. Let $f:M\to M$ be a $C^{1+\beta}$ diffeomorphism on a closed
smooth Riemannian manifold $M$ of dimension $d$, and let $\chi>0$.

\medskip
\noindent
{\sc $\chi$--hyperbolic measure:} An $f$--invariant probability measure $\mu$ on $M$
is called {\em $\chi$--hyperbolic measure} if $\mu$--almost surely all of its 
Lyapunov exponents are in $\mathbb R\backslash[-\chi,\chi]$.

\medskip
We restate Theorem \ref{theo: Markov partitions} below.

\begin{theorem}[\cite{Sa13,Ova16}]
Let $f:M\to M$ be a $C^{1+\beta}$ diffeomorphism. For each $\chi>0$, there exists a locally compact topological Markov shift $(\hSigma,\hsigma)$ and a H\"older continuous map $\hpi:\hSigma\rightarrow M$ such that: \begin{enumerate}[{\rm (1)}]
    \item $\hpi\circ\hsigma=f\circ\hpi$.
    \item $\hpi[\hSigma^\#]$ has full measure for every $\chi$-hyperbolic measure.
    \item Every $x\in \hpi_\chi[\hSigma^\#]$ has finitely many pre-images in $\hSigma^\#$. More specifically:
    there is $\varphi:\mathfs R\to\mathbb N$ such that if $x=\hpi(\hx)$ with $R_n=R$ for infinitely many
    $n>0$ and $R_n=S$ for infinitely many $n<0$ then $\#\{\hy\in \hSigma^\#:\hpi(\hy)=x\}\leq \varphi(R)\varphi(S)$.
    \item For every $\chi$-hyperbolic measure $\mu$, there exists a $\hsigma$-invariant measure $\hmu$
    such that $\hpi_*\hmu=\mu$.
    \end{enumerate}
\end{theorem}

\medskip
Given $r>0$, let $R[r]:=[-r,r]^d\subset\mathbb R^d$. We recap some definitions of \cite{Sa13}.
In the sequel, we fix $\varepsilon>0$ sufficiently small. Let $x\in M$ be a {\em regular point}
in the sense of the Oseledets theorem, and assume that the Lyapunov exponents of $f$ at $x$
are in $\mathbb R\backslash[-\chi,\chi]$. The {\em Pesin chart of $f$ at $x$}
is a map $\Psi_x:R[Q_\varepsilon(x)]\to M$, where $Q_\varepsilon(x)>0$ is a parameter
depending on the non-uniform hyperbolicity of $f$ at $x$. See \cite[Definition 1.1.12]{Ova16}.
In Pesin charts, the map $f$ becomes uniformly hyperbolic, see \cite[Theorem 1.1.13]{Ova16}. Given $\eta\leq Q_\varepsilon(x)$,
let $\Psi_x^\eta:R[\eta]\to M$ denote the restriction of $\Psi_x$ to $R[\eta]$.

\medskip
\noindent
{\sc Double Pesin chart:} A {\em double Pesin chart} is a pair $\Psi_x^{p^s,p^u}=(\Psi_x^{p^s},\Psi_x^{p^u})$
of Pesin charts with the same center $x$. See \cite[Section 4.1]{Sa13}.

\medskip
The parameters $p^s,p^u$ represent definite choices of sizes for
the stable and unstable manifolds of $x$. We draw an edge from $\Psi_x^{p^s,p^u}$ to
$\Psi_y^{q^s,q^u}$, and write $\Psi_x^{p^s,p^u}\to \Psi_y^{q^s,q^u}$, if the following
nearest neighbor conditions hold:
\begin{enumerate}[$\circ$]
\item The non-uniform hyperbolicities at $f(x)$ and $y$ are close, as well as the non-uniform
hyperbolicities at $f^{-1}(y)$ and $x$; formally, these are expressed
in terms of $\varepsilon$--overlaps between Pesin charts, see \cite[Definition 3.1]{Sa13} or
condition (GPO1) in \cite[Section 4]{LS16}.
\item The parameters $p^s,p^u,q^s,q^u$ are ``as large as possible''; formally, these are
expressed in terms of two greedy recursions, see condition (GPO2)
in \cite[Section 4]{LS16}.
\end{enumerate}

\medskip
\noindent
{\sc Generalized pseudo-orbit:} A {\em generalized pseudo-orbit {\rm (}gpo{\rm )}} is a sequence
$\{v_n\}_{n\in\mathbb Z}$ of double Pesin charts s.t. $v_n\to v_{n+1}$ for all $n\in\mathbb Z$.
A {\em positive gpo} is a sequence $\{v_n\}_{n\geq 0}$ of double Pesin charts s.t.
$v_n\to v_{n+1}$ for all $n\geq 0$; a {\em negative gpo} is a sequence $\{v_n\}_{n\leq 0}$
of double Pesin charts s.t. $v_{n-1}\to v_n$ for all $n\leq 0$. See \cite[Section 4]{LS16}.

\medskip
Let $\Psi_x^{p^s,p^u}$ be a double Pesin chart. An {\em $s$--admissible manifold} at $\Psi_x^{p^s,p^u}$
is the image under of $\Psi_x^{p^s,p^u}$ of the graph of a $C^{1+\beta/3}$ function that satisfies 
some regularity properties, see \cite[Section 4]{LS16} or \cite[Definition 1.3.10]{Ova16}. Similarly,
we define {\em $u$--admissible manifolds} at $\Psi_x^{p^s,p^u}$.
Using that $f$ is uniformly hyperbolic in Pesin charts,
we can introduce two graph transforms for each edge $v\to w$:
the {\em stable graph transform}  $\mathfs F^s_{v,w}$ sends $s$--admissible manifolds at $w$ to $s$--admissible
manifolds at $v$; the {\em unstable graph transform} $\mathfs F^u_{v,w}$ sends $u$--admissible
manifolds at $v$ to $u$--admissible manifolds at  $w$.
These maps are contractions, see \cite[Proposition 4.14]{Sa13}. Applying them along
positive and negative gpo's and passing to the limit, we define the
stable/unstable manifold of a positive/negative gpo.

\medskip
\noindent
{\sc Stable/unstable manifold of positive/negative gpo:} The {\em stable manifold}
of a positive gpo $\{v_n\}_{n\geq 0}=\{\Psi_{x_n}^{p^s_n,p^u_n}\}_{n\geq 0}$
is the $s$--admissible manifold at $v_0$,
$$
V^s[\{v_n\}_{n\geq 0}]=\{x\in \Psi_{x_0}(R[p^s_0]):f^k(x)\in \Psi_{x_k}(R[10Q_\varepsilon(x_k)]),\forall k\geq 0\},
$$
equal to the unique limit point obtained by the application of the sequence of contractions
$\mathfs F^s_{v_0,v_1}\circ\cdots\circ\mathfs F^s_{v_{n-1},v_n}$, $n\geq 1$.
The {\em unstable manifold}
of a negative gpo $\{v_n\}_{n\leq 0}=\{\Psi_{x_n}^{p^s_n,p^u_n}\}_{n\leq 0}$
is the $u$--admissible manifold at $v_0$,
$$
V^u[\{v_n\}_{n\leq 0}]=\{x\in \Psi_{x_0}(R[p^u_0]):f^k(x)\in \Psi_{x_k}(R[10Q(x_k)]),\forall k\leq 0\},
$$
equal to the unique limit point obtained by the application of the sequence of contractions
$\mathfs F^u_{v_{-1},v_0}\circ\cdots\circ\mathfs F^u_{v_{n-1},v_{n}}$, $n\leq 0$.
See \cite[Proposition 4.15]{Sa13}.

\medskip
We note that $V^s[\{v_n\}_{n\geq 0}]$ and $V^u[\{v_n\}_{n\leq 0}]$ are
stable and unstable manifolds in the sense of Pesin, see \cite[Proposition 6.3]{Sa13}.  

\medskip
\noindent
{\sc Shadowing:} Each gpo $\{v_n\}_{n\in\mathbb Z}$
{\em shadows} a unique point, equal to the intersection of the admissible manifolds
$V^s[\{v_n\}_{n\geq 0}]$ and $V^u[\{v_n\}_{n\leq 0}]$.
If $\{v_n\}_{n\in\mathbb Z}=\{\Psi_{x_n}^{p^s_n,p^u_n}\}_{n\in\mathbb Z}$,
this intersection is the unique point $x\in M$ s.t. $f^n(x)\in \Psi_{x_n}(R[Q(x_n)])$ for all $n\in\mathbb Z$.
See \cite[Theorem 4.16]{Sa13} and \cite[Theorem 4.2]{LS16}.

\medskip
The next step is to pass from the set of all double Pesin charts
to a countable family $\mathfs A$ s.t. generic points of $\chi$--hyperbolic measures are shadowed
by gpo's in $\mathfs A^{\mathbb Z}$. In the original work of Sarig, this is
in \cite[Proposition 3.5]{Sa13}. See also \cite[Theorem 5.1]{LM16} for a cleaner version.
The underlying idea used to prove these results is simple:
the set of all double Pesin charts is the union of countably many precompact subsets. Being precompact,
each of these subsets has a contable and dense subset. The union of these
coutable and dense subsets defines the countable family $\mathfs A$.
Let $\Sigma$ be the topological Markov shift with vertex set $\mathfs A$ and edge relation defined as above.

\medskip
\noindent
{\sc The coding $\pi$:} 
It is the map $\pi:\Sigma\to M$, $\pi[\{v_n\}_{n\in\mathbb Z}]:=V^s[\{v_n\}_{n\geq 0}]\cap V^u[\{v_n\}_{n\leq 0}]$.

\medskip
\noindent
{\sc The family $\mathfs Z$:} It is the family $\mathfs Z=\{Z_v:v\in\mathfs A\}$, where
$$
Z_v:=\{\pi(\underline v):\underline v=\{v_n\}_{n\in\mathbb Z}\in\Sigma^\#\text{ s.t. }v_0=v\}.
$$

\medskip
In general, each $Z_v$ is neither closed nor open.
Although the elements of $\mathfs Z$ might intersect non-trivially, the family
$\mathfs Z$ is locally finite: for each $Z\in\mathfs Z$, the set $\{Z'\in\mathfs Z:Z\cap Z'\neq\emptyset\}$
is finite. See \cite[Theorem 10.2]{Sa13}. This property is certainly one of the main difficulties
encountered in \cite{Sa13}. Let $Z\in\mathfs Z$, say $Z=Z(v)$.

\medskip
\noindent
{\sc Invariant fibres on $\mathfs Z$}: Given $x\in Z$, let
$W^s(x,Z):=V^s[\{v_n\}_{n\geq 0}]\cap Z$ and $W^u(x,Z):=V^u[\{v_n\}_{n\leq 0}]\cap Z$
for some (any) $\underline v=\{v_n\}_{n\in\mathbb Z}\in\Sigma^\#$
s.t. $\pi(\underline v)=x$ and $v_0=v$. See \cite[Definition 10.3]{Sa13}.

\medskip
Apply a {\em Bowen-Sina{\u\i} refinement} to the family $\mathfs Z$,
see \cite[Section 11.1]{Sa13}. This defines a new family $\mathfs R$ of disjoint sets that
covers the same set as $\mathfs Z$. The family $\mathfs R$ has three main properties, as we will now explain:
local finiteness with respect to $\mathfs Z$, product structure, and the Markov property.
Local finiteness with respect to $\mathfs Z$ follows by the local finiteness of $\mathfs Z$
and is expressed by two properties:
\begin{enumerate}[$\circ$]
\item For each $R\in\mathfs R$, the set $\{Z\in\mathfs Z:Z\supset R\}$ is finite.
\item For each $Z\in\mathfs Z$, the set $\{R\in\mathfs R:R\subset Z\}$ is finite.
\end{enumerate}
See \cite[Section 7.1]{LM16}. To define the product structure and Markov property
on $\mathfs R$, we need the following definition.

\medskip
\noindent
{\sc Invariant fibres on $\mathfs R$:} given $R\in\mathfs R$ and $x\in R$, define
$W^s(x,R):=W^s(x,Z)\cap R$ and $W^u(x,R):=W^u(x,Z)\cap R$. See \cite[Definition 11.4]{Sa13}.

\medskip
The sets $W^{s/u}(x,R)$ are usually fractal-like subsets of $R$, contained in the Pesin stable/unstable
manifolds at $x$. Fix $R\in\mathfs R$, and let $x,y\in R$. Since $W^s(x,R)$ and $W^u(y,R)$
are subsets of an $s$--admissible and a $u$--admissible manifold at a same double Pesin chart,
their intersection consists of at most one element. The {\em product structure} property
states that $W^s(x,R)$ and $W^u(y,R)$ indeed intersect at one point, denoted by $[x,y]$,
and that $[x,y]\in R$. See \cite[Proposition 11.5]{Sa13}.

\medskip
\noindent
{\sc Markov property:} Let $R,S\in\mathfs R$, and assume that $x\in R$ and $f(x)\in S$.
Then
$$
f[W^s(x,R)]\subset W^s(f(x),S)\text{ and }f^{-1}[W^u(f(x),S)]\subset W^u(x,R).
$$

\medskip
See \cite[Proposition 11.7]{Sa13}. Let $\widehat\Sigma$ be the topological Markov shift with
vertex set $\mathfs R$ and edge relation $R\to S$ iff $f(R)\cap S\neq\emptyset$. 

\medskip
\noindent
{\sc The coding $\widehat\pi$:} It is the map $\widehat\pi:\widehat\Sigma\to M$ defined
by the equality
$$
\{\widehat\pi(\underline R)\}:=\bigcap_{n\geq 0}\overline{f^n(R_{-n})\cap\cdots\cap f^{-n}(R_n)},
\ \text{ for }\underline R=\{R_n\}_{n\in\mathbb Z}\in\widehat\Sigma.
$$
This map is well-defined for two reasons: by the Markov property, the right-hand side is
the intersection of a descending chain of closed sets; by the non-uniform hyperbolicity,
the diameters decrease exponentially fast. See \cite[Section 12.2]{Sa13} for details.  
A point of attention is that, since each $R\in\mathfs R$ is usually neither closed nor open,
the good definition of $\widehat\pi$ requires that we take closures. This could potentially
increase the image of $\widehat\pi$ and prevent finiteness-to-one, but fortunately there is a relation
between the images of $\widehat\pi$ and $\pi$, as observed in \cite[Lemma 12.2]{Sa13}.
The proof of Lemma \ref{l.product} will use a stronger version of \cite[Lemma 12.2]{Sa13},
as we will now state. Given a finite path $v_n\to\cdots\to v_m$ on $\Sigma$, let
$$
Z_n[v_n,\ldots,v_m]:=\{\pi(\underline w):\underline w\in\Sigma^\#\text{ and }w_n=v_n,\ldots,w_m=v_m\}.
$$
Similarly, given a finite path $R_n\to\cdots\to R_m$ on $\widehat\Sigma$, let
$$
_{n}[R_n,\ldots,R_m]=\bigcap_{k=n}^m f^{-k}(R_k)=\{x\in M: f^n(x)\in R_n,\ldots,f^m(x)\in R_m\}.
$$

\noindent
{\sc Claim:} For each $\underline R=\{R_n\}_{n\in\mathbb Z}\in\widehat\Sigma$ and
$Z_v\supset R_0$,
there exists $\underline v=\{v_n\}_{n\in\mathbb Z}\in\Sigma$ with $v_0=v$
s.t. $_{n}[R_n,\ldots,R_m]\subset Z_n[v_n,\ldots,v_m]$ for every $n\leq m$.
In particular, $\widehat\pi(\underline R)=\pi(\underline v)$.

\begin{proof}[Proof of the claim.]
The difference of the above claim to \cite[Lemma 12.2]{Sa13} is that we fix some $Z_v\supset R_0$
and require that $v_0=v$. The proof is similar to the proof of \cite[Lemma 12.2]{Sa13}, but we
include the details for completeness.
Fix $\underline R=\{R_n\}_{n\in\mathbb Z}$.
For each $k\geq 0$, fix some $y_k\in~_{-k}[R_{-k},\ldots,R_k]$. Since $\mathfs R$ and $\mathfs Z$
cover the same subset of $M$, there is a gpo
$\underline v^{(k)}=\{v^{(k)}_\ell\}_{\ell\in\mathbb Z}$ with $v^{(k)}_0=v$ s.t. $y_k=\pi[\underline v^{(k)}]$.
Since the degrees of the graph defining $\Sigma$ are finite,
for each $\ell\geq 0$ there are finitely many possibilities for the tuple $(v^{(k)}_{-\ell},\ldots,v^{(k)}_{\ell})$.
By a diagonal argument, there is a gpo $\underline v=\{v_\ell\}_{\ell\in\mathbb Z}$
s.t. for each $\ell\geq 0$ the equality $(v_{-\ell},\ldots,v_\ell)=(v^{(k)}_{-\ell},\ldots,v^{(k)}_{\ell})$
holds for infinitely many $k\geq 0$. We will show that $\underline v$ satisfies the claim.

\medskip
Clearly $v_0=v$. For the other statements, firstly note that if $k\geq |n|$ then
$f^n(y_k)=\pi[\sigma^n(\underline v^{(k)})]$,
and so $R_n\subset Z_{v^{(k)}_n}$. In particular $R_n\subset Z_{v_n}$. Now let
$n\leq m$, and take $k\geq |m|,|n|$. Let $y\in~_{n}[R_n,\ldots,R_m]$. We wish to show
that $y\in Z_n[v_n,\ldots,v_m]$.
\begin{enumerate}[$\circ$]
\item Since $f^n(y)\in R_n\subset Z_{v_n}$, there is $\underline u\in\Sigma^\#$ with $u_0=v_n$ s.t.
$f^n(y)=\pi(\underline u)$.
\item Since $f^m(y)\in R_m\subset Z_{v_m}$, there is $\underline w\in\Sigma^\#$ with $w_0=v_m$
s.t. $f^m(y)=\pi(\underline w)$.
\end{enumerate}
Define $\underline a=\{a_\ell\}_{\ell\in\mathbb Z}$ by:
$$
a_\ell=\left\{
\begin{array}{ll}
u_{\ell-n}&, \ell\leq n\\
v_\ell&, n\leq \ell\leq m\\
w_{\ell-m}&, \ell\geq m.
\end{array}
\right.
$$
Since both $\underline u,\underline w\in\Sigma^\#$, also $\underline a\in\Sigma^\#$. Observe that:
\begin{enumerate}[$\circ$]
\item If $\ell\leq n$ then $f^\ell(y)=\pi[\sigma^{\ell-n}(\underline u)]\in Z_{u_{\ell-n}}=Z_{a_\ell}$.
\item If $n\leq \ell\leq m$ then $f^\ell(y)\in R_\ell\subset Z_{v_\ell}=Z_{a_\ell}$.
\item If $\ell\geq m$ then $f^\ell(y)=\pi[\sigma^{\ell-m}(\underline w)]\in Z_{w_{\ell-m}}=Z_{a_\ell}$.
\end{enumerate}
This shows that the gpo $\underline a$ shadows $y$, so $\pi(\underline a)=y$.
In particular $y\in Z_n[v_n,\ldots,v_m]$. To prove the equality $\widehat\pi(\underline R)=\pi(\underline v)$,
firstly observe that
$$\{\widehat\pi(\underline R)\}=\bigcap_{n\geq 0}\overline{_{-n}[R_{-n},\ldots,R_n]}\subset
\bigcap_{n\geq 0}\overline{Z_{-n}[v_{-n},\ldots,v_n]}.
$$
Since
$$
\{\pi(\underline v)\}=\bigcap_{n\geq 0}Z_{-n}[v_{-n},\ldots,v_n]\subset
\bigcap_{n\geq 0}\overline{Z_{-n}[v_{-n},\ldots,v_n]},
$$
the intersection $\bigcap_{n\geq 0}\overline{Z_{-n}[v_{-n},\ldots,v_n]}$ is non-empty.
But it consists of a descending chain of closed sets
with diameter converging to zero, so it equals $\{\pi(\underline v)\}$.
It follows that $\widehat\pi(\underline R)=\pi(\underline v)$.
\end{proof}

\begin{proof}[Proof of Lemma $\ref{l.product}$.]
Let $\underline R=\{R_n\}_{n\in\mathbb Z},
\underline S=\{S_n\}_{n\in\mathbb Z}\in\widehat\Sigma$ with $R_0=S_0=R$. Let
$x=\widehat\pi(\underline R)$ and $y=\widehat\pi(\underline S)$. We have
$[\underline R,\underline S]=\underline U$, where $\underline U=\{U_n\}_{n\in\mathbb Z}$
is defined by 
$$
U_n=\left\{
\begin{array}{ll}
R_n&, n\geq 0\\
S_n&, n\leq 0.
\end{array}
\right.
$$
We wish to show that $\widehat\pi([\underline U])=[x,y]$. Fix some 
$Z_v\supset R$. By the claim, there are gpo's
$\underline v=\{v_n\}_{n\in\mathbb Z},\underline w=\{w_n\}_{n\in\mathbb Z}$ with $v_0=w_0=v$
s.t. $\pi(\underline v)=x$ and $\pi(\underline w)=y$. Let $V^s:=V^s[\{v_n\}_{n\geq 0}]$
and $V^u:=V^u[\{w_n\}_{n\leq 0}]$, then $V^s\cap V^u=\{[x,y]\}$.

\medskip
Now define $\underline u=\{u_n\}_{n\in\mathbb Z}$ by 
$$
u_n=\left\{
\begin{array}{ll}
v_n&, n\geq 0\\
w_n&, n\leq 0.
\end{array}
\right.
$$
Clearly $\widehat\pi(\underline U)=\pi(\underline u)$, since:
\begin{enumerate}[$\circ$]
\item If $n\geq 0$ then $U_n=R_n\subset Z_{v_n}=Z_{u_n}$.
\item If $n\leq 0$ then $U_n=S_n\subset Z_{w_n}=Z_{u_n}$.
\end{enumerate}
But $V^s[\{u_n\}_{n\geq 0}]=V^s$ and $V^u[\{u_n\}_{n\leq 0}]=V^u$, so
$\{\pi(\underline u)\}=V^s\cap V^u=\{[x,y]\}$. Hence $\widehat\pi([\underline U])$ and $[x,y]$ both
equal the unique intersection between $V^s$ and $V^u$.
\end{proof}

The last part of the work of Sarig \cite{Sa13} consists of establishing the finiteness-to-one property
of $\widehat\pi$. Fix $x\in \widehat\pi[\widehat\Sigma^\#]$.
The original statement \cite[Theorem 12.8]{Sa13} stated wrongly that $\widehat\pi^{-1}(x)$ is finite,
but what is actually proved there is the finiteness of the intersection
$\widehat\pi^{-1}(x)\cap\widehat\Sigma^\#$. This inaccuracy, pointed out in \cite[Appendix A]{LS16},
does not cause problems for applications, since
$\widehat\Sigma\setminus\widehat\Sigma^\#$ does not contain any periodic orbits and 
has zero measure for every shift invariant probability
measure (Poincar\'e recurrence theorem). To conclude this appendix, we explicitly state the bound on
the cardinality of $\widehat\pi^{-1}(x)\cap\widehat\Sigma^\#$.

\medskip
\noindent
{\sc Affiliation:} We say that $R,R'\in\mathfs R$ are {\em affiliated} if there are $Z,Z'\in\mathfs Z$
s.t. $R\subset Z$, $R'\subset Z'$ and $Z\cap Z'\neq\emptyset$.

\medskip
Given $R\in\mathfs R$, define
$$
N(R):=\#\{(R',Z')\in\mathfs R\times\mathfs Z:R,R'\text{ are affiliated and }Z'\supset R'\}.
$$
The local finiteness properties of $\mathfs R$ with respect to $\mathfs Z$ imply that $N(R)<\infty$.
Let $x=\widehat\pi(\underline R)$ with $\underline R\in\widehat\Sigma^\#$. Sarig proved in 
\cite[Theorem 12.8]{Sa13} that if $R_n=R$ for infinitely many $n\geq 0$ and
$R_n=S$ for infinitely many $n\leq 0$, then $\widehat\pi^{-1}(x)\cap\widehat\Sigma^\#$ has
at most $N(R)N(S)$ elements. The proof is an adaptation of the ``diamond argument'' of Bowen
\cite[pp. 13--14]{Bo78}.

\bibliographystyle{plain}

\begin{thebibliography}{10}

\bibitem{AL}
J.F. Alves, V. Pinheiro.
\newblock  Gibbs-Markov structures and limit laws for partially hyperbolic attractors with mostly 
expanding central direction.
\newblock {\em  Adv. Math.}, 223 (2010), no. 5, 1706--1730.

\bibitem{AvV1}
A.~Avila and M.~Viana.
\newblock Simplicity of {L}yapunov spectra: a sufficient criterion.
\newblock {\em Port. Math.}, 64:311--376, 2007.

\bibitem{BBB}
L.~Backes, A.~Brown, and C.~Butler.
\newblock Continuity of {L}yapunov exponents for cocycles with invariant
  holonomies.
\newblock Preprint http://arxiv.org/pdf/1507.08978v2.pdf.

\bibitem{BP}
L.~Backes and M. Poletti.
\newblock Continuity of Lyapunov exponents is equivalent to continuity of Oseledets subspaces.
\newblock {\em Stochastics and Dynamics}, 17:1750047, 2017.

\bibitem{BaP07}
L.~Barreira and Ya. Pesin.
\newblock {\em Nonuniform hyperbolicity: dynamics of systems with nonzero
  {L}yapunov exponents}.
\newblock Cambridge University Press, 2007.


\bibitem{Ova16}
S.~{Ben Ovadia}.
\newblock {Symbolic dynamics for non uniformly hyperbolic diffeomorphisms of
  compact smooth manifolds}.
\newblock {Preprint https://arxiv.org/abs/1609.06494}.

\bibitem{Be2} M. Bessa,
\newblock Dynamics of generic multidimensional linear differential systems,
\newblock {\em Adv. Nonlinear Stud.}, 8, 191--211, 2008.

\bibitem{BeVar}
M. Bessa, P. Varandas, 
\newblock Trivial and simple spectrum for $SL(2,\mathbb{R})$ cocycles with free base and fiber dynamics.
\newblock \emph{Acta Math. Sinica,} 31, 7, 1113-1122, 2015.

\bibitem{BBCMVX}
M. Bessa, J. Bochi, M. Cambra\'inha, C. Matheus, P. Varandas and D. Xu,
\newblock Positivity of the top Lyapunov exponent for cocycles on semisimple Lie groups over hyperbolic bases,
\newblock Preprint https://arxiv.org/abs/1611.10158.

\bibitem{Boc02}
J.~Bochi.
\newblock Genericity of zero {L}yapunov exponents.
\newblock {\em Ergod. Th. {\&} Dynam. Sys.}, 22:1667--1696, 2002.

\bibitem{BGV03}
C.~Bonatti, X.~G{\'o}mez-Mont, and M.~Viana.
\newblock G\'en\'ericit\'e d'exposants de {L}yapunov non-nuls pour des produits
  d\'eterministes de matrices.
\newblock {\em Ann. Inst. H. Poincar\'e Anal. Non Lin\'eaire}, 20:579--624,
  2003.

\bibitem{BoV04}
C.~Bonatti and M.~Viana.
\newblock Lyapunov exponents with multiplicity 1 for deterministic products of
  matrices.
\newblock {\em Ergod. Th. {\&} Dynam. Sys}, 24:1295--1330, 2004.

\bibitem{Bo75}
R.~Bowen,
\newblock Equilibrium states and ergodic theory of Anosov diffeomorphism,
\newblock \emph{Lect. Notes in Math}, Springer Verlag, (1975).

\bibitem{Bo78}
R.~Bowen,
\newblock On {A}xiom {A} diffeomorphisms,
\newblock \emph{Regional Conference Series in Mathematics, No. 35}, American Mathematical Society (1978).

\bibitem{DK_b}
P.~Duarte and S.~Klein.
\newblock {\em Lyapunov exponents of linear cocycles}.
\newblock Atlantis series in Dynamical Systems. Springer.
\newblock To appear.

\bibitem{Fur63}
H.~Furstenberg.
\newblock Non-commuting random products.
\newblock {\em Trans. Amer. Math. Soc.}, 108:377--428, 1963.

\bibitem{GM89}
I.~Ya. Gol'dsheid and G.~A. Margulis.
\newblock Lyapunov indices of a product of random matrices.
\newblock {\em Uspekhi Mat. Nauk.}, 44:13--60, 1989.

\bibitem{GR86}
Y.~Guivarc'h and A.~Raugi.
\newblock Products of random matrices : convergence theorems.
\newblock {\em Contemp. Math.}, 50:31--54, 1986.

\bibitem{Le84a}
F.~Ledrappier.
\newblock Propri{\'e}t{\'e}s ergodiques des mesures de {S}ina{\"\i}.
\newblock {\em Publ. Math. I.H.E.S.}, 59:163--188, 1984.

\bibitem{LY}
F. Ledrappier and L.-S. Young.
 \newblock  The metric entropy of diffeomorphisms. I. Characterization of measures satisfying Pesin's entropy formula.
\newblock  {\em Ann. of Math.} 122 (3)  (1985), 509--539.

\bibitem{LS82}
F. Ledrappier and J.-M. Strelcyn.
   \newblock {A proof of the estimation from below in Pesin's entropy formula}
   \newblock{\em "Ergod. Th {\&} Dynam. Sys} 2 (1982) 203--219.

\bibitem{LS16}
Y.~Lima and O.~Sarig.
 \newblock Symbolic dynamics for three dimensional flows with positive topological entropy 
\newblock  To appear in J. Eur. Math. Soc.


\bibitem{LM16}
Y.~Lima and C.~Matheus.
 \newblock  Symbolic dynamics for non-uniformly hyperbolic surface maps with discontinuities.
\newblock  To appear in Ann. Sci. \'Ec. Norm. Sup\'er.


\bibitem{MMY15}
C. Matheus, M. M{\"o}ller, and J.-C. Yoccoz.
\newblock A criterion for the simplicity of the Lyapunov spectrum of
  square-tiled surfaces.
\newblock {\em Inventiones mathematicae}, 202(1):333--425, 2015.

\bibitem{New70}
S.~Newhouse.
\newblock Nondensity of {A}xiom {A}(a) on ${S}^2$.
\newblock In {\em Global analysis}, volume XIV of {\em Proc. Sympos. Pure Math.
  (Berkeley 1968)}, pages 191--202. Amer. Math. Soc., 1970.

\bibitem{Ose68}
V.~I. Oseledets.
\newblock A multiplicative ergodic theorem: {L}yapunov characteristic numbers
  for dynamical systems.
\newblock {\em Trans. Moscow Math. Soc.}, 19:197--231, 1968.


\bibitem{PeS83}
Ya. Pesin, and Ya. Sinai.
\newblock Gibbs measures for partially hyperbolic attractors.
\newblock  {\em Ergod. Th. $\&$ Dynam. Sys.} 3-4 (1983), 417--438.

\bibitem{Po16}
M.~{Poletti}.
\newblock {Stably positive Lyapunov exponents for $SL(2,\mathbb{R})$ linear
  cocycles over partially hyperbolic diffeomorphisms}.
\newblock Preprint https://arxiv.org/abs/1605.00044.

\bibitem{PoV16}
M.~{Poletti} and M.~{Viana}.
\newblock {Simple Lyapunov spectrum for certain linear cocycles over partially
  hyperbolic maps}.
\newblock Preprint https://arxiv.org/abs/1610.05294.

\bibitem{PSh89}
C.~Pugh and M.~Shub.
\newblock Ergodic attractors.
\newblock {\em Trans. Amer. Math. Soc.}, 312:1--54, 1989.

\bibitem{Sa13}
O.~Sarig.
\newblock Symbolic dynamics for surface diffeomorphisms with positive entropy.
\newblock {\em J. Amer. Math. Soc.}, 26:341--426, 2013.

\bibitem{Sm67}
S.~Smale.
\newblock Differentiable dynamical systems.
\newblock {\em Bull. Am. Math. Soc.}, 73:747--817, 1967.

\bibitem{Almost}
M.~Viana.
\newblock Almost all cocycles over any hyperbolic system have nonvanishing
  {L}yapunov exponents.
\newblock {\em Ann. of Math.}, 167:643--680, 2008.

\bibitem{LLE}
M.~Viana.
\newblock {\em Lectures on {L}yapunov {E}xponents}.
\newblock Cambridge University Press, 2014.

\bibitem{FET}
M.~Viana and K.~Oliveira.
\newblock {\em Foundations of Ergodic Theory}.
\newblock Cambridge University Press, 2015.

\end{thebibliography}

\end{document}